\documentclass[12pt, reqno]{amsart}

\usepackage{amssymb, amsfonts, amsthm, amsmath, bbm, xcolor}
\usepackage{a4wide}
\usepackage{mathtools}
\usepackage{comment}
\usepackage{cite}
\usepackage{nicefrac}
\usepackage{graphicx}
\usepackage[foot]{amsaddr}

\usepackage{enumerate}
\newtheorem{theorem}{Theorem}[section]

\newtheorem{lemma}[theorem]{Lemma}

\newtheorem{claim}[theorem]{Claim}
\newtheorem{corollary}[theorem]{Corollary}
\newtheorem{remark}[theorem]{Remark}

\newtheorem{conjecture}[theorem]{Conjecture}

\newtheorem*{theorem*}{Theorem}
\usepackage{etoolbox}
\patchcmd{\thmhead}{(#3)}{#3}{}{}

\newcommand{\eps}{\varepsilon}

\begin{document}

\author{Jordan Chellig} 
\address{School of Mathematics, University of Birmingham B15 2TT, Birmingham, United Kingdom.}  \email{jac555@student.bham.ac.uk}
 \author{Calina Durbac}
 \address{Department of Computer Science, University of Durham, DH1 3LE, Durham, United Kingdom.} \email{calina.m.durbac@durham.ac.uk}
  \author{Nikolaos Fountoulakis}
 \address{School of Mathematics, University of Birmingham B15 2TT, Birmingham,  United Kingdom.}\email{n.fountoulakis@bham.ac.uk}

%
%

\title[Best response dynamics on Random Graphs]{Best response dynamics on random graphs}

\thanks{This research started during the Undergraduate Research Bursary 17-18 91 (for Calina Durbac under the supervision of Nikolaos Fountoulakis) 
funded by the London Mathematical Society and the School of Mathematics in Birmingham. 
Jordan Chellig is supported by an EPSRC scholarship. 
Nikolaos Fountoulakis has also been partially supported by an Alan Turing Fellowship EP/N5510129/1
and the EPRSC grant EP/P026729/1.}

\begin{abstract}
We consider evolutionary games on a population whose underlying topology of interactions is determined by 
a binomial random graph $G(n,p)$. Our focus is on 2-player symmetric games with 2 strategies played between the incident members of such a population. Players update their strategies synchronously. At each round, each player selects the strategy that is the best response to the current set of strategies its neighbours play. 
We show that such a system reduces to generalised majority and minority dynamics. We show rapid 
convergence to unanimity for $p$ in a range that depends on a certain characteristic of the payoff matrix. 
In the presence of a bias among the pure Nash equilibria of the game, 
we determine a sharp threshold on $p$ above which the largest connected component reaches unanimity with high probability. 
For $p$  below this critical value, where this does not happen, we identify those 
substructures inside the largest component that remain discordant throughout the evolution of the system. 
\end{abstract}
\maketitle

\section{Introduction}

The need for a dynamic game theory was pointed out by von Neumann and Morgenstern in their seminal book~\cite{bk:NeumannMorgen} which set the foundations of modern game theory. Research on dynamic games on populations was stimulated by settings in evolutionary biology. In 1973 John Maynard Smith and George Price~\cite{ar:SmithPrice1973} set the foundations of \emph{evolutionary game theory}, trying to explain phenomena that arise in animal fighting, which contradict 
the traditional Darwinian theory. 

 Maynard Smith and Price~\cite{ar:SmithPrice1973} considered a variation of the well-known hawk-dove game (see below for a precise description of its payoff matrix) with more than two strategies, which range from purely altruistic to hawkish behaviour. They argued that altruistic behaviour in such a population may have higher payoff 
 compared to the payoff in a pure population consisting only of hawks. Introducing the notion of an \emph{evolutionary stable strategy}, they argued that natural selection would favour the co-existence of behaviours in a given population. 
In this setting, every member of the population may interact with any other member. That is, in graph-theoretic terms, this is an evolutionary system on a complete graph. 

Later on,~Nowak and May~\cite{ar:NowMay1992} considered settings with non-trivial underlying topology. They considered a population of agents that are located on  the vertices of a 2-dimensional lattice, in which every agent interacts only 
with its four neighbours. The interaction is that any two adjacent agents play the prisoner's dilemma (which we 
describe below) in which they may \emph{cooperate} or \emph{defect}. The dynamics that was considered there 
depends on the total payoff of each agent that is accumulated by the four games it plays with its neighbours (or 
three or two, if it is located on the boundary of the lattice). The agents update their strategies synchronously, with 
each agent adopting the strategy of a neighbour who had the largest total payoff in the previous round. 
The main observation in~\cite{ar:NowMay1992} is the \emph{co-existence} of the two strategies in the long term, 
despite the fact that the 
defect strategy is a \emph{Nash equilibrium} for this game: that is, in a game between two players who both defect none has interest in cooperating, given that the other does not. Similar dynamics were studied by Santos and Pacheko~\cite{ar:SantPach2005} in the preferential attachment model, which typically yields graphs which have some properties of complex networks. 
 
In this paper, we consider a class of dynamics of games on networks in which agents seek to 
choose a~\emph{local dominating strategy}. In this process, every agent switches their strategy to the one which 
maximises their own payoff, given the strategies of its opponents. In other words, in each round each agent 
tries to give their \emph{best response} to the current choice of strategies of their opponents. This dynamics was 
first considered by Gilboa and Matsui~\cite{ar:GilbMat91} in  continuous-time  setting and for a population 
with no underlying topology, where everyone interacts with everyone else. Gilboa and Matsui showed the 
existence of~\emph{cyclically stable sets}. Roughly speaking, these are sets of configurations of the 
population in which the best response dynamics is trapped. 

%


We consider a discrete setting and study the evolution of best response dynamics on a random 
graph. Adjacent agents are able to interact with each other via a symmetric game with 2 strategies, on a round-by-round basis. In a given round, each agent will choose one of the two (pure) strategies, which they will use against all of their neighbours. For each game played, the agent will receive a payoff determined by the corresponding entry in the payoff matrix; subsequently, each agent will also receive a total payoff, which is the sum of all of the individual payoffs that the agent received in that round. In the next round, agents will choose to execute the strategy which would have given them the largest payoff in the previous round. In essence, each agent is focused on maximising their own total payoff by utilising information gained from the previous round. Our central result concerns the rapid rate at which agents will settle into a global unanimous strategy on a typical sample of the underlying random graph. 

In our analysis, we consider arbitrary 2-player symmetric\footnote{Here, the term ``symmetric" means that the roles of the two players are interchangable.} games and we show that best response dynamics reduce to \emph{generalised majority or minority dynamics}. These terms refer to a general class of discrete processes on graphs 
where vertices have two states and at each round a vertex adopts the state of the majority or the minority of its 
neighbourhood, respectively. 
Majority games reward the individuals who follow strategies which are in line with the popular opinion; such games are clearly co-operative, and thus agents will tend to form large unanimous coalitions; see \cite{hofbauer2003evolutionary}. On the other hand, minority games capture the idea that agents will benefit from making choices that oppose the popular consensus. The idea of a minority game was introduced to characterise the behaviour of agents within the \emph{El Farol Bar problem}. The problem describes a fixed population who will repeatedly attempt to synchronously choose their favourite evening location; however, only the individuals who choose less crowded locations will be rewarded~\cite{de1999competition, chakrabarti2009kolkata}. This problem captures the underlying tension of minority games. Choosing the lucrative option will clearly reward agents with the greatest payoff, however this payoff is easily spoiled if too many of the participants think alike. If agents adopt this line of thinking, then we deduce that the entire population will reject the optimal payoff, regardless of the fact that the most lucrative option will tend to be uncontested. In the case of the El Farol Bar problem, any deterministic strategy will ultimately fail to satisfy anyone. 

 In the context of market investments, majority and minority strategies are of central importance. As suggested in \cite{marsili2001market}, traders utilising contrarian-like strategies will view the market functioning in a minority-game-like way. They believe goods on the market have a fundamental price and will invest in such a way to discourage the market from deviating away from its fundamental value. On the other hand, investors who are \emph{trend followers} view the market from the context of a majority game. These investors tend to inflate the price of goods which are already travelling on an upward trend. Such investment strategies are believed to be the central cause driving the \emph{pricing bubble} phenomena \cite{challet2001minority,kets2007minority,bk:minority}. The success of each of these strategies is heavily dependent on the payoff agents can expect to receive, along with the motives and actions of their rivals. 


While the study of evolutionary games on populations with arbitrary mutual interactions is well established, a new focus has now been given to systems of agents which possess an underlying topology. Within these systems, agents are only able to interact with their topological neighbours. The topology of these interactions is commonly represented by means of an underlying network. Consideration of an underlying topology allows us to analyse how local decisions in the system can cascade out to form a global consensus. For example, in~\cite{lelarge2011diffusion} it is shown that a small set of agents which oppose the current consensus, can cause a large contagion of opposing opinions to spread throughout the network. While in~\cite{blume1993statistical} systems of interacting agents on a lattice are considered, agents have the opportunity to asynchronously update their opinion at random times, which are driven by an underlying Poisson point process. In the windows of opportunity which are given by these random times, agents will choose to update their strategies to the one which will now give them the largest current payoff. It is shown that a steady state can be achieved within this system; the shape of this distribution is heavily dependent on how agents choose to update their opinion during these random times. 

The underlying topology which is the main focus of this paper is that of a binomial random graph $G(n,p)$. 
This classic model has been central in the study of random graphs and was first considered by Gilbert~\cite{ar:Gilbert59}. 
In this model, on a set of $n$ vertices each edge appears independently with probability $p=p(n)$. 
We shall consider the evolution of the best response dynamics on $G(n,p)$ for an arbitrary symmetric game 
with two players played on the vertex set of a $G(n,p)$ random graph. As we shall see in our analysis, best response dynamics will reduce to the analysis of generalised majority or minority dynamics on $G(n,p)$. 

Our results show that in a wide range of densities best response dynamics stabilises rapidly. 
We prove a general result which shows that this is achieved in at most 4 rounds when $np = \Omega (n^{1/2})$. 
However, if the game exhibits a certain form of bias in terms of its pure Nash equilibria, then we get very 
precise results on how this dynamics evolves for $p$ such that $np \gg 1$. Let us recall that in this range 
$G(n,p)$ may not be connected. However, typically 
its largest connected component contains almost all vertices, 
whereas each one of the other components is a tree and has order at most $\log n$. (The references~\cite{bk:BollobasRG,bk:JLR2000} provide a detailed description of the typical structure and the 
evolution of $G(n,p)$.)
In the presence of such a bias, we determine a sharp threshold on $p$ above which the largest connected component reaches consensus among its vertices. For $p$ below this critical value, we identify those 
substructures inside the largest component that are in disagreement with the majority of the vertices therein. 
Hence, we are able to characterise the co-existence of strategies very precisely. 

We layout our paper as follows: In Section \ref{Defn} we provide key definitions for our analysis and define the \emph{interacting node system} on a graph. In Section \ref{Prelimnary_Analysis}, we consider our first approaches to analysing the model and state our two main results. The first concerns the formation of a consensus when the underlying networks is both random and suitably dense. The second result shows that a consensus can be achieved in sparser a regime of $G(n,p)$, given that there exists a certain form of bias among the pure Nash equilibria of the game. 
This is analysed in Section~\ref{Skewed_Node_Systems}. 
In this case, we identify a sharp threshold for $p$ above which consensus is reached in the largest component of 
$G(n,p)$ in at most $\beta \log n$ rounds, for a positive constant $\beta$, where $\log \cdot$ is the natural logarithm. 
This critical value is determined by the payoff matrix and is below the well-known connectivity threshold of $G(n,p)$ which is $\log n/n$. 
Furthermore, 
we show that if $np > c \log n$, for some $c>1$ which depends on the parameters of the system, then in fact the 
system reaches consensus only after one round. 

In Section \ref{Unskewed_Node_Systems}, we consider games where this bias is no longer present. In this scenario, we observe that the interacting node system reduces to the so-called \emph{majority or minority dynamics} on $G(n,p)$. We sample a selection of results from~\cite{fountoulakis2019resolution}, which concern the rapid stabilisation of agent strategies when playing a majority game on a dense random graph. Using these results as a basis, we proceed to prove our result concerning the formation of consensus strategies in the minority game analogue of the random graph majority game. By combining the above results, we may readily deduce that for any $2 \times 2$ real-valued payoff matrix and suitably dense random graph, the agents of the interacting node system will reach a consensus after at most four rounds, with high probability. 

\section{Iterative Games on Graphs}\label{Defn}
Let $Q = (q_{i,j})$ be a $2 \times 2$ matrix with entries in $\mathbb{R}$; we refer to $Q$ as the \textit{payoff matrix}. (The rows and columns of $Q$ are indexed by $\{0 , 1\}$, which are assumed to represent the strategies.)  Each player will now choose a strategy from $\{0 , 1 \}$. Player 1 will then receive a payoff given by $q_{i,j},$ where $i$ is the strategy chosen by Player 1, and $j$ is the strategy chosen by Player 2. 
Analogously, Player 2 will receive a payoff $q_{j,i}$.

The interactions of the agents/players are represented by a fixed graph $G= (V,E)$. 
The nodes of $G$ represent the agents, and if two nodes are adjacent, then the corresponding agents interact with each other by playing the game with payoff matrix $Q$. We refer to this process as an \textit{interacting node system} $(G , Q , \mathcal{S})$, which we detail formally as follows. We fix a graph $G,$ a payoff matrix $Q,$ and for every $v \in V(G)$ an initial vertex strategy $\mathcal{S}(v)$ where $\mathcal{S}: V(G) \rightarrow \{0 , 1\}$. We consider a discrete time process. For each $t \in \mathbb{N}_0 := \mathbb{N} \cup \{0\}$, we denote by $S_{t}(v)$  the strategy played by vertex $v$ at step $t$. Thus, $S_{t}: V(G) \rightarrow \{0, 1\}$.
Let $N_G(v)$ denote the set of neighbours of a vertex $v \in V(G)$ in $G$.
For a vertex $v\in V(G)$, a step $t \in \mathbb{N}_0$ and $j \in \{0 , 1 \},$ we set $n_{t}(v ; j) = \left\lvert N_{G}(v) \cap \{u : S_{t}(u) = j \} \right\vert$. To each vertex $v \in V(G),$ we assign the initial state $S_{0}(v) = \mathcal{S}(v)$. To progress from round $t$ to round $t+1$, the following evolution rule is applied: Each vertex $v$ will play the game against each one of its neighbours, playing strategy $S_{t}(v)$. For each game played, the vertex will receive a payoff given by the corresponding entry in the payoff matrix $Q.$ We denote the \emph{total payoff} for a vertex $v$ at time step $t$ to be the sum of all payoffs that $v$ received in that round. Therefore, for a vertex $v$ with $S_{t}(v) = i,$ we define the \textit{total payoff} of $v$ at time $t$ as $ T_{t}(v) := n_{t}(v; 0)q_{i , 0} + n_{t}(v; 1)q_{i , 1}.$ We define the \textit{alternative payoff} of $v$ at time $t$, as the total payoff the vertex would have received, had they played the other available strategy. Hence, the alternative payoff for a vertex $v$ with $S_{t}(v) = i$ is given as: $T'_{t}(v) := n_{t}(v; 0)q_{1 - i , 0} + n_{t}(v; 1)q_{1 - i , 1}$. The value of $S_{t+1}(v)$ is determined by comparing the values of the current payoff and the alternative payoff. If the alternative payoff is strictly greater than the total payoff, then the vertex will switch to strategy $ 1-i$ in round $t+1$; otherwise it will continue with its current strategy. For a vertex $v$ at time $t,$ we can succinctly express the evolution rule as follows:

$$ S_{t+1}(v) = \begin{cases}
S_{t}(v)  & \textrm{if} \ T_{t}(v) \geq T'_{t}(v);\\
 1 - S_{t}(v)   & \textrm{if} \ T_{t}(v) < T'_{t}(v).\\
\end{cases}
$$

We wish to analyse the global evolution of the strategies as time elapses. We say that a vertex's strategy is \textit{periodic} if there exist some $T , p \in \mathbb{N},$ such that $S_{T}(v) = S_{T + kp}$ for all $k \in \mathbb{N}$. We call the least possible $p$ which satisfies this definition the \emph{period}. We say that the evolution of a system is \textit{unanimous}, if there exists some $T \in \mathbb{N}$ such that for all $t \geq T$ and all distinct pairs of vertices $u, v \in V(G),$ we have that $S_{t}(u) = S_{t}(v).$ We say that the system is \textit{stable} if there exists $T \in \mathbb{N}$ such that for all $v \in V(G)$ and $t \geq T,$ we have that $S_{T}(v) = S_{t}(v).$

\section{Main results}\label{Prelimnary_Analysis}

The evolution of the node system is largely governed by the entries present in the payoff matrix $Q.$ As a consequence of choosing specific entries for $Q,$ it can be the case that one strategy is strictly more beneficial than the other. In this case the evolution of the system is trivial, as all agents will unanimously maximise their payoff; we capture this behaviour with the notion of degeneracy. We say that a payoff matrix $Q = (q_{i,j})$ with $ i,j \in \{0,1 \}$ is \textit{non-degenerate}, if one of the following holds:

\begin{enumerate}[(i)]
    
    \item We have that $q_{0,0} > q_{1,0}$ and $q_{0,1} < q_{1,1}.$
    
    \item Or, $q_{0,0} < q_{1,0}$ and $q_{0,1} > q_{1,1}.$
    
\end{enumerate}
Otherwise, we say that $Q$ is~\emph{degenerate}. The behaviour of interacting node systems with degenerate payoff matrices will be discussed in Section~\ref{Degeneracy}.

We consider an interacting node system where the underlying graph $G$ is both random and suitably dense. We 
let $G(n,p)$ be the \textit{binomial random graph} on vertex set $V_n := [n] := \{1,2 \ldots , n \}$, where 
each edge appears independently with probability $p$. We introduce a random binomial initial state, which we refer to as $\mathcal{S}_{1/2} \in \{0 , 1\}^{[n]}.$ For 
every vertex $v \in V_{n},$ we have that  $\mathbb{P}[\mathcal{S}_{1/2}(v) = 1] =\mathbb{P}[\mathcal{S}_{1/2}(v) = 
0] = 1/2$, independently of any other vertex. 
We say that a sequence of events $E_{n}$ defined on a sequence of probability spaces with probability measure 
$\mathbb{P}_n$ occurs \emph{asymptotically almost surely} (or \emph{a.a.s.}) if 
$\mathbb{P}_{n}[E_{n}] \to 1$ as $n\to \infty$.

We now state our first result which describes the vertex strategies of 
an interacting node system on $G(n,p),$ with initial state $\mathcal{S}_{1/2}$. We denote this system by 
$(G(n,p),Q,\mathcal{S}_{1/2})$.

\begin{theorem}\label{main_result}
Let  $Q$ be a  $2 \times 2$ non-degenerate payoff matrix.  For any  $\varepsilon \in (0, 1]$ there exist positive constants $\Lambda, n_{0}$ such that for all $n \geq n_{0}$, if $p > \Lambda n^{-\frac{1}{2}},$ then with probability at least $1 - \varepsilon,$ across the product space of $G(n,p)$ and $\mathcal{S}_{1/2}$, the interacting node system $(G(n,p), Q , \mathcal{S}_{1/2})$ will be unanimous after at most four rounds.
\end{theorem}

\subsection{The skew of the payoff matrix $Q$}
Suppose a vertex $v$ at time $t$ has picked a strategy $S_{t}(v) = i$. From the above discussion, we observe that the \textit{incentive} for $v$ to switch strategy is the condition that $T_{t}(v) < T'_{t}(v).$ Expanding both terms and re-arranging, we observe the condition for a vertex playing strategy $i$ to switch to strategy $ 1 - i $ is as follows:
$$n_{t}(v; 0) \left(q_{i , 0} - q_{ 1 - i, 0} \right) <  n_{t}(v; 1)\left(q_{ 1 - i , 1} - q_{i , 1}\right).$$
If $Q$ is non-degenerate, there are two possible cases: Either $q_{0,0} > q_{1,0}$ and $q_{1,1} > q_{0,1}$; or we have that $q_{0,0} < q_{1,0}$ and $q_{0,1} > q_{1,1}$.

We lead with the former case. By substituting values of $i \in \{0 , 1 \},$ we can rephrase the evolution conditions for each agent in terms of $n_{t}(v;0)$ and $n_{t}(v;1).$ We define the \textit{payoff skew} of the payoff matrix $Q$ as: 
$$\lambda = \lambda(Q) :=\left(q_{1,1} - q_{0,1}\right) / \left(q_{0,0} - q_{1,0} \right).$$ 
We remark that $\lambda(Q)$ is positive and well-defined, if and only if, $Q$ is non-degenerate. 
If $S_{t}(v) = i$, then we can write:
\begin{equation} \label{eq:majority}
S_{t+1}(v) = \begin{dcases}
 1 - i   & \textrm{if} \ n_{t}(v;i) <  \lambda^{1-2i} n_{t}\left(v; 1 - i \right) \\
 i   & \textrm{otherwise} \\ 
\end{dcases}.
\end{equation}
Suppose now that the latter case holds, that is, $ q_{0,0} < q_{1,0}$ and $q_{0,1} > q_{1,1}.$ 
Then if $S_{t}(v) = i$,
\begin{equation} \label{eq:minority} 
S_{t+1}(v) = \begin{dcases}
1 - i    & \textrm{if} \ n_{t}(v;i) >  \lambda^{1-2i} n_{t}\left(v;  1 - i \right)\\
 i   & \textrm{otherwise} \\
\end{dcases}.
\end{equation}

 We refer to the system governed by the evolution rules from \eqref{eq:majority} as the \textit{majority regime}; while we refer to the system described in \eqref{eq:minority} as the \textit{minority regime}. We will tackle each of these systems separately. The intuition for these names arises from how each individual agent tends to think of its neighbours. In the majority regime, agents will tend to  follow the strategies which are shared by the majority of their neighbours; however, in the minority regime agents will generally choose the least popular strategy seen across their neighbours. The value of $\lambda$ determines the strength of this tendency.  
 
 In game-theoretic terms, it is easy to see that payoff matrices in the majority regime give rise to two Nash equilibria which are pure strategies: namely with both players playing simultaneously strategy 1 or strategy 0. 
 On the other hand, payoff matrices in the minority regime have two Nash equilibria which are pure strategies 
 where the two players play opposite strategies. Hence, games of the former type are sometimes called \emph{coordination games}, whereas those of the latter type are called~\emph{anti-coordination games}.
 From this point of view, the parameter $\lambda$ can be seen as some form of bias between the two pure Nash equilibria.
 
 \subsubsection*{Examples} We give two examples of games that are well studied in game theory that reflect the above classification.

 \noindent
 1. \emph{The Hawk-Dove game.} In this game, a population of animals consists of two types of individuals which are differentiated by the amount of aggression they display during their interactions. There is the most aggressive type (hawk) and the least aggressive, or most cooperative, type (dove). When two of them interact over some fixed resource, 
 the outcome depends on their types. If two hawks compete over the resource they get injured, because of the fighting, and one of them (at random) manages to grab the resource. Hence, if $R$ is the gain from the resource, 
 each is expected to gain $R/2$ out of the fight, but they pay a price $P >R$ for their injuries, whereby their overall 
 gain is $(R-P)/2$. If a hawk interacts with a dove, then the hawk grabs the resource gaining $R$, whereas the dove walks away with nothing.  Finally, if two doves interact, then they share the resource each gaining $R/2$. 
Thus, if $Q_{\mathrm{H-D}}$ denotes the payoff matrix of the Hawk-and-Dove game, this is 
$$Q_{\mathrm{H-D}} = \left[ 
\begin{array}{cc}
\frac{R-P}{2} & R \\
0 & \frac{R}{2}
\end{array}
\right], $$
where first row and column correspond to the hawk strategy and the second row and column correspond to the dove strategy. 
Hence, this is a non-degenerate payoff matrix in the minority regime with 
$$\lambda (Q_{\mathrm{H-D}})= \frac{R}{P-R}.$$

\noindent
2. \emph{The Prisoners dilemma.} In this game, two individuals are arrested while committing a crime together, but they are put in different cells. The police do not have enough evidence to convict them, but they make an offer to each one separately. If one of them confesses (defects) but the other remains silent (cooperates), then the former is released but the other is sentenced to imprisonment of $P>0$ years. 
If they both confess, they are sentenced to $R>0$ years of imprisonment, for $R<P$.
Finally, if both remain silent, they are both sentenced to $S$ years imprisonment  for some minor offence, as the police do not have enough evidence. In this case, $S<R$. 
 
Thus, $0< S <R <P$. The payoff matrix of the Prisoners dilemma game $Q_{\mathrm{PD}}$ is 
$$Q_{\mathrm{PD}} = \left[ 
\begin{array}{cc} 
-R & 0 \\
-P & -S
\end{array}
\right],  $$
where the first row and column correspond to the Defect strategy and the second row and column correspond to the Cooperate strategy. (Here, we put the $-$ sign in front of these quantities, as imprisonment is thought of as a
loss.) Hence, this is a case of a degenerate payoff matrix.

Our second result concerns non-degenerate interacting node systems in a sparser regime compared to that covered in Theorem~\ref{main_result}.
In this setting, we assume that $\lambda \neq 1$ and
 we are able to consider smaller values of $p$, below the \emph{connectivity threshold} of $G(n,p)$. 
Among the first results on the theory of random graphs due to Gilbert~\cite{ar:Gilbert59} is that the connectivity of $G(n,p)$ undergoes a sharp transition.  More specifically, let $\omega : \mathbb{N} \to \mathbb{R}_+$ 
be a function such that $\omega (n) \to \infty $ as $n \to \infty$; 
if $np = \log n + \omega (n)$, then a.a.s. $G(n,p)$ is connected, 
whereas if $np = \log n - \omega (n)$, then a.a.s. $G(n,p)$ has at least one isolated vertex. Analogous results were proved 
by Erd\H{o}s and R\'enyi~\cite{ar:ErRen59} on the uniform model of random graphs with a given number of edges.

In the subcritical regime, we cannot hope for unanimity. For $1 \ll np< \log n$, all connected components
of $G(n,p)$ besides the largest one are trees of order at most $\log n$ and, in fact, many of them are isolated 
vertices
(see for example~\cite{bk:BollobasRG} for a complete and precise description of the structure of $G(n,p)$).  
\begin{theorem}\label{Spare_skewed_Stabilistion}
Let  $p =d/ n \leq 1$, where $d \gg 1$, 
and let $Q$ be a $2 \times 2$ non-degenerate payoff matrix. Suppose that $(G(n,p), Q, \mathcal{S}_{1/2})$ is an interacting node system with payoff skew $\lambda \neq 1$. For any $\eps >0$ there exists $\beta = \beta (\lambda, \eps)>0$ such that 
a.a.s. at least $n(1-\eps )$ vertices in $G(n,d/n)$ will be unanimous after at most $\beta \log n$ rounds. 

Moreover, there exists a constant $ \alpha(\lambda)>1$ such that if $d> \alpha(\lambda) \log n$, then a.a.s. 
$G(n,d/n)$ will be unanimous after one round. 
\end{theorem}
The part of the above theorem for $np> c \log n$ with $c > \alpha(\lambda)$ implies and, in fact, it is much stronger than Theorem~\ref{main_result}. That is, for $\lambda \not = 1$ the system goes into unanimity in one round for 
much lower densities than $n^{-1/2}$. 
Thus, it will  remain to prove Theorem ~\ref{main_result} only for $\lambda = 1$. 

We refine the above theorem focusing on the largest connected component of $G(n,p)$ 
 which we denote as $L_1 (G(n,p))$ (formally, if there are at least two, we take the lexicographically smallest one).  
Let $u_n^{(1)}$ be the probability that $L_1(G(n,d/n))$ will eventually become unanimous. 
The next two theorems give the precise location of the threshold on $p$ above which $u_n^{(1)}$ approaches 1. 
In fact, there are two different thresholds for the two regimes. We start with the majority regime. 
For $\lambda > 0$, let 
$$\ell_\lambda := \lceil \max \{\lambda, \lambda^{-1} \} \rceil \ \mbox{and} \ 
c_\lambda := \frac{1}{\ell_\lambda +1}. $$ 
(For a real number $x>0$, we let $\lceil x \rceil = x$, if $x\in \mathbb{N}$, and $\lceil x\rceil = \lfloor x \rfloor +1$, otherwise.)
We write $\mathrm{Po}(\gamma)$ for the Poisson distribution with parameter $\gamma> 0$. 
\begin{theorem} \label{thm:Skewed_Majority}
Suppose that $\lambda \not =1$ and $d=c_\lambda \log n + \log \log n + \omega (n)$. 
Then the following hold in the majority regime.
\begin{enumerate}
\item If  $\omega (n) \to +\infty$ as $n\to +\infty$, then 
$$\lim_{n\to +\infty} u_n^{(1)} = 1.$$
\item If $\omega (n) \to c \in \mathbb{R}$ as $n\to +\infty$, then
$$\limsup_{n\to +\infty} u_n^{(1)} \leq \sum_{k=0}^{\infty} \left(1- \frac{1}{2^{\ell_\lambda+1}}\right)^k
\mathbb{P} \left(\mathrm{Po} (e^{c(\ell_\lambda +1)}/ \ell_\lambda!) = k \right).$$
and 
$$\liminf_{n\to +\infty} u_n^{(1)} \geq \sum_{k=0}^{\infty}  \left(\frac{1}{2^{\ell_\lambda+1}}\right)^k
\mathbb{P} \left(\mathrm{Po} (e^{c(\ell_\lambda +1)}/ \ell_\lambda!) = k \right). $$
\item If $\omega (n) \to - \infty$ as $n \to +\infty$, then 
$$\lim_{n\to +\infty} u_n^{(1)} = 0.$$
\end{enumerate}
\end{theorem}
The analogous result for the minority regime is as follows. For $\lambda >0$, we let 
$$ \ell_\lambda' = \lfloor \max \{\lambda, \lambda^{-1} \}\rfloor. $$
\begin{theorem} \label{thm:Skewed_Minority}
Suppose that $\lambda \not =1$ and $d=\frac{1}{2} \log n + \frac{1+\ell_\lambda'}{2}\log \log n + \omega (n)$. 
Then the following hold in the minority regime.
\begin{enumerate}
\item If  $\omega (n) \to +\infty$ as $n\to +\infty$, then 
$$\lim_{n\to +\infty} u_n^{(1)} = 1.$$
\item If $\omega (n) \to c \in \mathbb{R}$ as $n\to +\infty$, then
$$\limsup_{n\to +\infty} u_n^{(1)} \leq \sum_{k=0}^{\infty} \left(\frac{3}{4}\right)^k
\mathbb{P} \left(\mathrm{Po} (e^{2c}/ \ell_\lambda'!) = k \right).$$
and 
$$\liminf_{n\to +\infty} u_n^{(1)} \geq 
\mathbb{P} \left(\mathrm{Po} (e^{2c}/ \ell_\lambda'!) = 0 \right). $$
\item If $\omega (n) \to - \infty$ as $n \to +\infty$, then 
$$\lim_{n\to +\infty} u_n^{(1)} = 0.$$
\end{enumerate}
\end{theorem}


The reason for the existence of two different thresholds is that the structures that block unanimity are different 
in the two regimes. Effectively, these are the thresholds for their disappearance as subgraphs of $G(n,p)$ and, 
as it turns out, as subgraphs of $L_1(G(n,p))$.
Nevertheless, as our arguments show, in both regimes the unanimity of $L_1(G(n,p))$ is achieved in $O(\log n)$ 
steps.

During our analysis, we will see that there are two different kinds of unanimity depending on whether the payoff matrix is in the majority or in the minority regime.  
In the majority regime, all vertices of $L_1(G(n,p))$ stabilise to one of the two strategies. However, 
in the minority regime, the vertices of $L_1 (G(n,p))$ arrive at unanimity but they fluctuate incessantly between the two strategies.  In other words, in the majority regime the subsystem of $L_1(G(n,p))$ becomes periodic with period 1, whereas in the minority regime the period is equal to 2. 
Furthermore, we identify the strategy that is played at each step once this subsystem has entered the periodic cycle.

We close this introductory section with a brief discussion on the case of degenerate payoff matrices. 

\subsection{Degenerate payoff matrices}\label{Degeneracy} If $Q$ is degenerate, then the behaviour of the interacting node system is readily deduced. The system will either reach stability from the outset, or it will reach  stable unanimity after one round. We summarise this behaviour in the following lemma. 

\begin{lemma}\label{Degenerate_System}
Let $Q$ be a degenerate payoff matrix, $G$ a connected graph, and $\mathcal{S}$ an initial configuration of vertex strategies. Then the interacting node system $(G,Q,\mathcal{S}),$ evolves as follows: Either the system is stable from $T = 0,$ or the system is unanimous and stable from $T = 1$.  
\end{lemma}

\begin{proof}

We divide our proof into a number cases. Firstly, we note that if $q_{0,0} = q_{1,0}$ and $q_{0,1} = q_{1,1},$ then we must achieve stability from the initial state,  since $T_{0}(v) = T'_{0}(v).$ Hence, for all $v \in V(G)$ we have that $S_{0}(v) = S_{t}(v)$ for $t \geq 0.$

Suppose that the above case does not occur, and we have $q_{0,0} > q_{1,0}.$ For $Q$ to be degenerate, we are forced to have that $q_{0,1} \geq q_{1,1}.$ Consequently, we now have top row domination in $Q.$ For any game played by vertex $v$, it is always optimal to play strategy zero. As $G$ is connected, every vertex will play at least one game, and, in particular, it is always optimal to play strategy 0, which corresponds to the top row of the payoff matrix. Therefore, for all $v$ we have that $S_{1}(v) = 0,$ and a stable unanimity is achieved. A similar row domination argument follows for the remaining possibilities of degenerate matrices.
\end{proof}

As Lemma \ref{Degenerate_System} holds when $G$ is any connected graph, it then follows that this argument suffices as a proof of both Theorem \ref{main_result} and Theorem \ref{Spare_skewed_Stabilistion} for the case that $Q$ is a degenerate payoff matrix. 

\subsection{Notation and probabilistic tools} We finish the introductory section giving some notation. For two distinct vertices $v,u \in V_n$ 
we write $v\sim u$ to indicate that they are adjacent. We also use the symbol $\sim$ in the context of random variables. In particular, for a random variable $X$ we will write $X \sim \mathrm{Bin} (s,q)$ to indicate that the distribution of $X$ is the binomial with parameters $s \in\mathbb{N}$ and $q \in [0,1]$. 
Furthermore, we write $f(n)\sim g(n)$ for two real-valued functions on $\mathbb{N}$, if $\lim_{n\to \infty} |f(n)|/|g(n)| =1$.

For a vertex $v$ in a given graph, we will set $d(v):= |N_G(v)|$ to be the degree of $v$ within the graph $G$, 
and recall that $N_G(v)$ denotes the neighbourhood of $v$ in $G$.
The particular graph we refer to will be made specific by the context.  
If $S\subset V$, we will be writing $e(S)$ for the number of edges that $G$ spans inside $S$. Also, if $v \in 
V$ we denote by $d_S(v)$ the degree of $v$ inside $S$.

When we consider the process on $G(n,p)$ we shall use the symbols $P_t$ and $N_t$ to denote the subsets of 
vertices whose agents play strategy 1 and 0, respectively, after the $t$th step. More formally,
for $t \geq 0$, we define $P_t = \{ v \in V_n : S_t(v) = 1 \}$ and $N_t = \{ v \in V_n : S_t(v) = 0 \}$. 

\subsubsection*{The Chernoff bound}
Throughout our arguments we will use the standard Chernoff bounds for concentration of binomially distributed 
random variables. The inequality we will use follows from Theorem 2.1 in~\cite{bk:JLR2000}.
If $X$ is a random variable such that  $X\sim \mathrm{Bin} (N,q)$, then for any $\delta \in (0,1)$ we have 
\begin{equation} \label{eq:Chernoff}
	\mathbb{P} (|X - Nq| \ge \delta Nq) \le 2 e^{-\delta^2 Nq / 3}.
\end{equation}

\subsubsection*{The FKG inequality}

We also require a correlation inequality that is known as the \emph{FKG inequality}. We state it in the context of 
the model $G(n,p)$. We say that a graph property $\mathcal{P}$ is \emph{non-decreasing} if, for graphs $G,H$ on $V_n$, whenever $G \in \mathcal{P}$ and $G \subseteq H$, then $H \in \mathcal{P}.$ Similarly we say that $\mathcal{P}$ is \emph{non-increasing}, if whenever $H \not  \in \mathcal{P}$ and $H \subseteq G$, then $G \not \in \mathcal{P}$ as well.  
We state the FKG inequality as follows.
\begin{theorem}[The FKG inequality~\cite{bk:AlonSpencer}] \label{thm:FKG}
Let $\mathcal{P}_{1}$ be an non-decreasing graph property and $\mathcal{P}_{2}$ be a non-increasing graph property. Then for the binomial random graph, $G(n,p),$ we have the following:

$$\mathbb{P}[G(n,p) \in \mathcal{P}_{1} \cap G(n,p) \in \mathcal{P}_{2}] \leq \mathbb{P}[G(n,p) \in \mathcal{P}_{1}]\cdot \mathbb{P}[G(n,p) \in \mathcal{P}_{2}].$$
\end{theorem}

\section{Interacting node systems with skew $\lambda \not =1$}
\label{Skewed_Node_Systems}

\subsection{Blocking structures and their distribution} 

 We will start with the identification of those structures/induced subgraphs of $G(n,p)$ which, roughly speaking, will stay immune to what the rest of the graph is doing. Thus, these substructures act as obstructions to unanimity.  

In particular, the structure we consider is an \emph{$(\ell,k)$-blocking star}. 
This is a star whose central vertex has degree $\ell +k$ in $G(n,d/n)$ and, furthermore, $\ell$ leaves of the star have degree 1 inside $G(n,d/n)$, whereas we impose no restriction on the degrees of the remaining $k$ leaves. We call the latter leaves the \emph{connectors} of the blocking star, whereas the $\ell$ leaves of degree 1  are called the \emph{blocking leaves}. 
Such a structure is illustrated in Figure~\ref{fg:blocking_star}.
\begin{figure}[h] 
\includegraphics[scale=1]{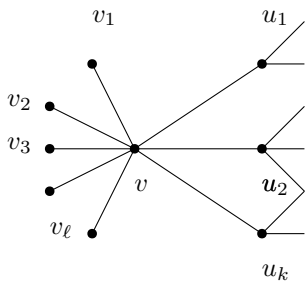}
\caption{A blocking star with central vertex $v$, blocking leaves $v_1,\ldots, v_\ell$ and connectors $u_1,\ldots , u_k$.}\label{fg:blocking_star}
\end{figure} 
An $(\ell,k)$-blocking star has the property that it can block or absorb the influence of the external vertices depending on the choice of its parameters $\ell$ and $k$, respectively. 
Let us consider first  the minority regime. It will turn out that it will be sufficient to consider 
$(1,k)$-blocking stars for a suitable defined $k$. 
Let $i^* \in \{0,1\}$ be such that $\lambda^{2i^*-1} = \max \{\lambda, \lambda^{-1}\}$. 
For $i,j\in \{0,1\}$, we say that a star has the \emph{$(i,j)$-configuration} if the blocking leaves play strategy $i$
whereas the centre plays strategy $j$. 
\begin{claim} \label{clm:1-k_sub}
Consider a $(1,k)$-blocking star with $k \leq \lfloor \lambda^{2i^*-1} \rfloor$. In the minority regime, 
if the star ever gets into the $(i^*,1-i^*)$-configuration, it will stay so forever. 
\end{claim}
\begin{proof} Let $v_1$ denote the blocking leaf and $v$ denote the centre of the star. 
Suppose that $S_t(v_1)=i^*$ but $S_t (v) = 1-i^*$. 
By~\eqref{eq:minority}, the vertex $v$ will change strategy, if $n_t(v;1-i^*) > \lambda^{2i^*-1} n_t (v;i^*)$. 
Since $S_t(v_1)=i^*$, then $n_t(v;1-i^*)\leq k  \leq \lfloor \lambda^{2i^*-1} \rfloor$. 
But also $n_t(v;i^*)\geq 1$. So we would have $ \lfloor \lambda^{2i^*-1} \rfloor > \lambda^{2i^*-1}$, which is 
impossible. Thus, $S_{t+1} (v) = 1-i^*$.  Furthermore, $S_{t+1} (v_1)=i^*$, since $v_1$ has no neighbours playing $i^*$. 
\end{proof}

\begin{claim} \label{clm:1-k_sup}
Consider a $(1,k)$-blocking star with $k > \lfloor \lambda^{2i^*-1} \rfloor$. In the minority regime, if the $k$
connectors simultaneously alternate between $1-i^*$ and $i^*$, then the blocking leaf and centre will eventually synchronise with them.
\end{claim}
\begin{proof}
Let $v_1$ be the blocking leaf, $v$ be the centre and $u_1,\ldots, u_k$ be $k$ connectors of the star. 
Suppose that $S_t (u_j) = 1-i^*$, for all $j=1,\ldots, k$. 

If $S_t (v) = 1-i^*$, then $v$ will change strategy if $n_t (v;1-i^*) > \lambda^{2i^*-1} n_t (v;i^*)$ (cf.~\eqref{eq:minority}). 
But $n_t (v;1-i^*) \geq  k$ and $n_t (v;i^*) \leq 1$. Since $k >  \lfloor \lambda^{2i^*-1} \rfloor$, the above inequality is indeed satisfied. Hence, $S_{t+1}(v) =i^*$. Also, note that $S_{t+1} (v_1) = i^*$. 

On the other hand, if $S_t (v) = i^*$, then $v$ will not change strategy if $n_t (v;i^*) \leq \lambda^{1-2i^*} 
n_t (v;1-i^*)$.  But $n_t (v;1-i^*) \geq  k>  \lfloor \lambda^{2i^*-1} \rfloor$. Therefore, 
$n_t (v; 1-i^*) > \lambda^{2i^*-1}$. So $\lambda^{1-2i^*} n_t (v;1-i^*) > 1$. But $n_t (v;i^*) \leq 1$ and the
inequality is satisfied. Therefore, $S_{t+1}(v)=i^*$. However, now $S_{t+1} (v) = 1-i^*$. 

Suppose now that $S_t (u_j) = i^*$, for all $j=1,\ldots, k$.
If $S_t (v) = 1-i^*$, then by~\eqref{eq:minority} $v$ will change strategy if $n_t (v;1-i^*) > \lambda^{2i^*-1} n_t (v;i^*)$. 
But now $n_t (v;1-i^*) \leq 1$ and $n_t (v;i^*) \geq k > \lambda^{2i^*-1}$. 
Thus, the above inequality is not satisfied and $S_{t+1}(v) =1-i^*$. Also, note that $S_{t+1} (v_1) = i^*$. 

If $S_t (v) = i^*$, then $v$ will not change strategy if $n_t (v;i^*) \leq \lambda^{1-2i^*} 
n_t (v;1-i^*)$. Now, $n_t(v;i^*) \geq k > \lambda^{2i^*-1} >1$ but $\lambda^{1-2i^*} 
n_t (v;1-i^*) \leq \lambda^{1-2i^*} < 1$. So the above inequality is not satisfied and $S_{t+1} (v) = 1-i^*$. 
Furthermore, $S_{t+1} (v_1)=1-i^*$.

We thus conclude that in any case the centre $v$ will synchronise with the $k$ connectors. 
Note that from the above four cases, we see that the blocking leaf $v_1$ will synchronise with 
$v$ at the steps where it changes state. Hence, it will also synchronise with the $k$ connectors too. 
\end{proof}
In majority regime, it will turn out that we will need to consider $(\ell,1)$-blocking stars. 
We will show that if $\ell$ is sufficiently large, then the $(\ell,1)$-blocking star can block the influence of the 
external vertices and retain the strategy of its vertices. 
\begin{claim} \label{clm:ell-1_maj}
Consider an $(\ell,1)$-blocking star with $\ell \geq \lceil \lambda^{2i^*-1} \rceil$. In the majority regime, 
if it is set to the $(1-i^*,1-i^*)$-configuration initially, then it will stay in this configuration forever. 
\end{claim}
\begin{proof}
Suppose that $v_1,\ldots, v_\ell$ are the $\ell$-blocking leaves and $v$ is the centre of the star. 
Assume that all $S_t (v) = S_t(v_1) = \cdots = S_t(v_{\ell})=1-i^*$. 
By~\eqref{eq:majority}, the centre will change strategy at step $t+1$, 
if $n_t(v;1-i^*) < \lambda^{2i^*-1} n_t (v;i^*)$. 
But $n_{t} (v;1-i^*) = \ell \geq \lceil \lambda^{2i^*-1} \rceil$ and $n_t (v;i^*) \leq 1$. 
Hence, we should have $ \lceil \lambda^{2i^*-1} \rceil < \lambda^{2i^*-1}$, which is impossible.
So $S_{t+1}(v) = 1-i^*$. 
Further, note that the $\ell$-blocking leaves will adopt the strategy of the centre at step $t+1$, since they 
have no other neighbours. Thus, $S_{t+1}(v_j) = 1-i^*$, for any $j=1,\ldots, \ell$, as well. 
\end{proof}
\begin{claim} \label{clm:ell-1_maj_synch}
Consider an $(\ell,1)$-blocking star with $\ell \geq 1$. In the majority regime, 
if it is set to the $(i^*,i^*)$-configuration initially, then it will stay in this configuration forever. 
\end{claim}
\begin{proof}
Suppose that $v_1,\ldots, v_\ell$ are the $\ell$-blocking leaves and $v$ is the centre of the star. 
Assume that all $S_t (v) = S_t(v_1) = \cdots = S_t(v_{\ell})=i^*$. 
By~\eqref{eq:majority}, the centre will change strategy at step $t+1$, 
if $n_t(v;i^*) < \lambda^{1-2i^*} n_t (v;1-i^*)$. 
But $n_{t} (v;i^*)  \geq 1$ and $n_t (v;1-i^*) \leq 1$. 
So the above inequality is not satisfied, since $\lambda^{1-2i^*}<1$, and, therefore, $S_{t+1}(v)=i^*$.  
Finally,  the $\ell$-blocking leaves will retain the strategy of the centre at step $t+1$, since they 
have no other neighbours. Thus, $S_{t+1}(v_j) = i^*$, for any $j=1,\ldots, \ell$, as well. 
\end{proof}


Now, we shall give a general condition that determines the distribution of 
the $(\ell,k)$-blocking stars in $G(n,d/n)$. The following results will be useful both for the subcritical and the 
supercritical regime that we analyse in the next subsection. 

Let $X_{\ell,k,n}$ be the random variable which is the number of $(\ell,k)$-blocking stars in $G(n,d/n)$ and 
let $X_{\ell,k,n}^{(1)}$ be the number of those ($\ell, k)$-blocking stars which are subgraphs of $L_1 (G(n,d/n))$. 
Clearly, $X_{\ell,k,n}^{(1)} \leq X_{\ell,k,n}$.  
However, we will show that a.a.s. these two random variables are approximately 
equal. 

In particular, we will show the following two lemmas. 
\begin{lemma} \label{lem:X_ell} Let $\ell\in \mathbb{N}$ and 
 $p= d/n$, where $1\ll d =d(n) =O(\log n)$.  Then $\mathbb{E} X_{\ell,n} \sim n \frac{d^{\ell+k}}{\ell!k!}  e^{-d(\ell+1)}$.
 
Furthermore, if  $d= \frac{1}{\ell+1} \log n + \frac{\ell+k}{\ell+1}\log \log n + \omega (n)$, then the following hold. 
\begin{enumerate}
\item[i.]
 If $\omega (n)\to -\infty$ as $n\to +\infty$, then $\mathbb{E} X_{\ell,k,n} \to +\infty$ as $n\to +\infty$ and  moreover,
 a.a.s. $X_{\ell,k,n} \geq \mathbb{E} X_{\ell,k,n} /2$.  
 \item[ii.]
 If $\omega(n) \to c \in \mathbb{R}$ as $n\to +\infty$, then 
 $\mathbb{E} X_{\ell,k,n} \to \frac{e^{(\ell+1)c}}{\ell!k!}$ as $n\to \infty$ and 
 $$X_{\ell,k,n} \stackrel{d}{\to} \mathrm{Po} \left(\frac{e^{(\ell+1)c}}{\ell!k!} \right). $$
 \item[iii.] 
 If $\omega (n)\to +\infty$ as $n\to +\infty$, then 
 $\mathbb{P} (X_{\ell,k,n} >0) <2 e^{-(\ell+1) \omega (n)}$, 
 for any $n$ sufficiently large.  Thus, a.a.s.  $X_{\ell,k,n} =0$.
\end{enumerate}
\end{lemma}
Let $X_{\ell,k,n}^{(2)} = X_{\ell,k,n} - X_{\ell,k,n}^{(1)}$.
\begin{lemma} \label{lem:X_ell1} 
We have $\mathbb{E} X_{\ell,k,n}^{(2)} = o(\mathbb{E} X_{\ell,k,n})$.
\end{lemma}
So Markov's inequality implies that a.a.s. $X_{\ell,k,n}^{(2)} \leq  \mathbb{E} X_{\ell,k,n} /4$. 
By Lemma~\ref{lem:X_ell}, if $d=\frac{1}{\ell +1} \log n + \frac{\ell+k}{\ell + 1} \log \log n + \omega (n) $, with 
$\omega (n) \to - \infty$, then 
a.a.s. 
\begin{equation}\label{eq:stars_in_giant}
X_{\ell,k, n}^{(1)} \geq \frac{1}{4} \mathbb{E} X_{\ell,k,n}.
\end{equation}
Furthermore, note that any two $(\ell,k)$-blocking stars can share only their connector vertices. So, if we consider 
the initial assignment of strategies, each $(\ell,k)$-blocking star inside $L_1 (G(n,d/n))$ 
will be set into $(i,j)$-configuration 
with probability $1/2^{\ell+1}$, independently of each other. 
Thus, the weak law of large numbers together with~\eqref{eq:stars_in_giant} imply the following. 
\begin{lemma} \label{lem:subcritical_description}
Let $p = d/n$, where $d = \frac{1}{\ell+1} \log n + \frac{\ell + k}{\ell+1}\log\log n + \omega (n)$, for $k,\ell \in \mathbb{N}$. 
Let $i,j \in \{0,1\}$.
\begin{enumerate}
\item[i.] If  $\omega (n) \to -\infty$ as $n\to +\infty$, then a.a.s. 
at least $\frac{1}{8} \cdot \frac{1}{2^{\ell+1}} \mathbb{E} X_{\ell,k ,n}$ of the $(\ell_{\lambda},k)$-blocking stars inside $L_1(G(n,d/n))$ will be set into $(i,j)$-configuration at the beginning of the process.

\item[ii.]If $\omega (n) \to c \in\mathbb{R}$ as $n\to +\infty$, then the number of $(\ell,k)$-blocking stars inside $L_1(G(n,d/n))$ will be set into $(i,j)$-configuration at the beginning of the process converges in distribution as $n\to \infty$ to a random variable distributed as $\mathrm{Bin} (\mathrm{Po} (e^{(\ell+1)c}/(\ell! k!)) ,  \frac{1}{2^{\ell+1}})$.
\end{enumerate}
\end{lemma}
We conclude this section with the proofs of Lemmas~\ref{lem:X_ell} and~\ref{lem:X_ell1}.
\begin{proof}[Proof of Lemma~\ref{lem:X_ell}]
We start with the expected value of $X_{\ell,k,n}$. Suppose $S$ denotes a set of size $\ell+k+1$ 
on which  an $(\ell,k)$-blocking star will be formed. 
There are $\binom{n}{\ell+k+1}$ ways to select these vertices 
and $(\ell+k+1) {\ell + k  \choose k}$ ways to select the centre $v$ and the connector vertices $u_1,\ldots, u_k$. 
Suppose that the remaining vertices are $v_1,\ldots, v_\ell$.  
We have $\mathbb{P} ({d}(v_1)=1 | v_1 \sim v) = (1-p)^{n-1}$. 
For $j=2,\ldots, \ell$, 
$$\mathbb{P} (d(v_j)=1 | v_i \sim v, d(v_i)=1, \mbox{for $i=1,\ldots, j-1$})= (1-p)^{n- (j-1)} < e^{-d + d\ell /n}.$$
But since $j \leq \ell$, we have for $n$ sufficiently large
$$ e^{-d - (d/n)^2}\leq (1-p)^{n- (j-1)} < e^{-d + d\ell /n}. $$
Hence, using the assumption that $d = O(\log n)$
$$ \mathbb{P}(d(v_1)=\cdots= d(v_{\ell})=1 | v_i \sim v, \mbox{for $i=1,\ldots, \ell$}) \sim e^{-d\ell}. $$
Also, $\mathbb{P} (d_{V_n\setminus S}(v) = 0) = (1-p)^{n - (\ell+k+1)} \sim e^{-d }$
Thus, we obtain 
\begin{eqnarray*}
\mathbb{E} X_{\ell,k,n} &\sim& \frac{n^{\ell+k+1}}{(\ell+k+1)!} \cdot (\ell+k+1) {\ell+k \choose k} \cdot \left(\frac{d}{n}\right)^{\ell+k} \cdot 
e^{-d (\ell+1)} \\ 
&=&\frac{n^{\ell+k+1}}{(\ell+k+1)!} \cdot (\ell+k+1) \frac{(\ell+k)!}{\ell! k!} \cdot \left(\frac{d}{n}\right)^{\ell+k} \cdot 
e^{-d (\ell+1)}\\
&=& n \frac{d^{\ell+k}}{\ell! k!} \cdot e^{-d(\ell+1)}. 
\end{eqnarray*}
This concludes the first part of the lemma. 

Now, the value of 
$\lim_{n\to \infty} \mathbb{E} X_{\ell,k,n}$ is deduced as in parts $i., ii.$ and $iii.$ of the lemma follows by taking
$d(\ell +1) = \log n + (\ell+k)\log \log n + (\ell+1) \omega(n)$, where either $\omega (n) \to +\infty$ or $\to  c$ 
or $\to -\infty$, as $n\to\infty$, respectively.  

For Part $iii.$, Markov's inequality implies that 
$$ \mathbb{P} (X_{\ell,k,n} > 0) \leq \mathbb{E} X_{\ell,k,n} \stackrel{n \ large}{<} 2 \cdot e^{- (\ell+1)\omega (n)} \to 0, \ \mbox{as $n\to +\infty$}.  $$ 

For Parts $i.$ and $ii.$ we will show that 
for any fixed integer $r\geq 2$, we have 
$\mathbb{E} (X_{\ell,k,n} (X_{\ell,k,n}-1) \cdots (X_{\ell,k,n} - (r-1)))  \sim \mathbb{E}^r X_{\ell,k,n}$. 

So the second statement in Part $i.$ will follow from Chebyshev's inequality as for $r=2$, the above 
implies that 
$\mathrm{Var} (X_{\ell,k,n}) = o(\mathbb{E}X_{\ell,k,n})$.  The second statement in Part $ii.$ will follow from Theorem~1.22 in~\cite{bk:BollobasRG} (Poisson convergence).

Consider $r$ subsets $S_1,\ldots, S_r\subset V_n$ of size $\ell+k+1$. They may all induce $(\ell,k)$-blocking stars only if any two of them share at most $k$ vertices. For $S \subset V_n$ let $I_S$ be the indicator random variable that is equal to 1 if and only if $S$ is an $(\ell,k)$-blocking star. 
If $S_i \cap S_j = \varnothing$ for all $i\not = j$, then 

\begin{equation} \label{eq:asymp_ind} 
\mathbb{P} (I_{S_1}= \cdots =I_{S_r}=1) \sim \left( (\ell+k+1) {\ell + k \choose k} p^{(\ell+k)} e^{-d(\ell+1)}\right)^r.
\end{equation}
There are $\prod_{i=1}^r \binom{n - (i-1)(\ell+k+1)}{\ell+k+1}$ to select the ordered $r$-tuple 
of pairwise disjoint sets $(S_1,\ldots, S_r)$. 
But 
\begin{equation} \label{eq:prod_coeff} 
\prod_{i=1}^r \binom{n - (i-1)(\ell+k+1)}{ \ell+k+1}  \sim \left(\frac{n^{\ell+k+1}}{(\ell+k+1)!}\right)^r.
\end{equation}

Now, let us assume that $S_i \cap S_j \not = \varnothing$ for some $i\not =j$.  
Note first that if $|S_i \cap S_j|>k$, then the two sets cannot induce $(\ell,k)$-blocking stars simultaneously. 
If $S_i \cap S_j = \{u_1,\ldots, u_s\}$, with $s\leq k$ and $S_i,S_j$ induce $(\ell,k)$-blocking stars, then 
$u_1,\ldots, u_s$ must be connector vertices in both of them. 
There are $O(1)$ ways to select the vertices in $\cup_{i=1}^r S_i$ that will be the connectors of the 
$(\ell,k)$-blocking stars. Let $S' \subset \cup_{i=1}^r S_i$ be such a choice; by the above observation, this 
contains any vertex which belongs to at least two members of the $r$-tuple. 
Set $S^{1..r} = \cup_{i=1}^r S_i$. 
Hence, 
$$\mathbb{P} (I_{S_1}=\cdots =I_{S_r}=1) = O(1)\cdot
\left[(1-p)^{{\binom{\ell +k+1}{ 2}}-(\ell+k)} p^{(\ell+k)} \right]^r \cdot \mathbb{P} (\forall v \in S^{1..r} \setminus S', 
d_{V_n \setminus  S^{1..r}}(v)=0). $$
Note that $|S^{1..r} \setminus S'| = r(\ell+1)$.
So, the latter probability is 
$$\mathbb{P} (\forall v \in S^{1..r} \setminus S', 
d_{V_n \setminus  S^{1..r}}(v)=0) = (1-p)^{r(\ell+1) (n- |S^{1..r}|)} \sim e^{-dr(\ell+1)}. $$
Therefore, 
\begin{equation} \label{eq:asymp_error}
\mathbb{P} (I_{S_1}=\cdots =I_{S_r}=1) = O(1)\cdot  p^{r(\ell+k)} e^{-dr(\ell+1)}. 
\end{equation}
Now, such an ordered $r$-tuple can be selected into at most 
\begin{equation} \label{eq:error_count}
{\binom{n}{ r(\ell+k+1)-1}}  = o(n^{r(\ell+k+1)}) 
\end{equation}
ways.
Thus, 
\begin{eqnarray*}
& &\mathbb{E} (X_{\ell,k,n} (X_{\ell,k,n}-1) \cdots (X_{\ell,k,n} - (r-1))) = \sum_{(S_1,\ldots, S_r): S_i \subset V_n, |S_i|=\ell+k+1} \mathbb{P}(I_{S_1}=\cdots=I_{S_r}=1) \\
&=& (1+o(1))  \left(\frac{n^{\ell+k+1}}{(\ell+k+1)!} \cdot  \left( (\ell+k+1) {\ell+k \choose k} \right) p^{\ell+k} e^{-d(\ell+1)} \right)^r 
\ \mbox{by~\eqref{eq:asymp_ind},\eqref{eq:prod_coeff}}\\
& &\hspace{2cm}
+o(n^{r(\ell+k+1)}) p^{r(\ell+k)} e^{-dr(\ell+1)} \ \mbox{by~\eqref{eq:asymp_error},\eqref{eq:error_count}}\\
&\sim& \mathbb{E}^r X_{\ell,k,n}. 
\end{eqnarray*} 
\end{proof}

\begin{proof}[Proof of Lemma~\ref{lem:X_ell1}] 
In  this lemma, we will bound the expected number of $(\ell,k)$-blocking stars which do not belong to 
$L_1 (G(n,d/n))$. Recall that the random variable which counts these is $X_{\ell,k,n}^{(2)}$. 

We will give an upper bound on the expected number of connected components of order at most $\log n$
which contain an $(\ell,k)$-blocking star. This suffices due to the following result about the structure of $G(n,d/n)$.
\begin{theorem} \label{clm:small_components} Let $p=d/n$ with $d \gg 1$. 
Then a.a.s. all connected components of $G(n,d/n)$ apart from 
$L_1 (G(n,d/n))$ have order at most $\log n$. 
\end{theorem}
This is a direct consequence of the proof of Theorem 6.10 (pp.143--146) in~\cite{bk:BollobasRG} and we refer the interested reader directly to this. 

For $r \geq \ell+k+1$, let $C_{\ell,k,r}$ denote the number of connected components which contain an $(\ell,k)$-blocking star and have order $r$. 
We will give an upper bound on the expected value of $C_{\ell,k,r}$. For two sets $S \subset S' \subset V_n$ having 
$|S|=\ell+k+1$ and $|S'| =r$, and $v, u_1,\ldots, u_k \in S$, we set $I(S,S',v,u_1,\ldots, u_k)$ to be the indicator random variable which is equal to 1 
if and only if $S'$ is a connected component in $G(n,d/n)$, where 
in particular, $S$ induces an $(\ell,k)$-blocking star with centre $v$ and connectors 
$u_1,\ldots, u_k$. 
Let $\mathcal{S}_{\ell,k,r}$ denote the set of pairs $(S,S')$ such that $S \subset S' \subset V_n$ with
$|S|=\ell+k+1$ and $|S'| =r$. Moreover, for a set $S\subset V_n$ of size at least $k+1$, let $S^{(k+1)}_{\not =}$ 
denote the set of all ordered $k+1$-tuples of distinct vertices in $S$.  
With this notation, we can write
\begin{equation} \label{eq:clk_exp}
\mathbb{E} (C_{\ell,k}) \leq \sum_{(S,S') \in \mathcal{S}_{\ell,k,r}} \sum_{(v,u_1,\ldots, u_k)\in S^{(2)}_{\not =}} 
\mathbb{E} (I(S',S,v,u_1,\ldots, u_k)). 
\end{equation} 
For $S \subset V_n$ having $|S|=\ell+k+1$ and $(v,u)\in S^{(k+1)}_{\not =}$, we 
let $I_1 (S,v,u_1,\ldots, u_k)$ be the indicator random variable that is equal to 1 if and only if 
$S$ forms an $\ell$-blocking star with centre $v$ and connectors $u_1,\ldots, u_k$. 
Also, we  take $I_2 (S')$ to be the indicator random variable that is equal to 1 if and only if 
$S'$ is a connected component of $G(n,d/n)$. 

For a given $(S,S') \in \mathcal{S}_{\ell,k}$ and $(v,u_1,\ldots, u_k)\in S^{(k+1)}_{\not =}$, we write
\begin{eqnarray}
& &\mathbb{P} (I(S',S,v,u_1,\ldots, u_k)=1) = \nonumber\\
& &\hspace{2cm} \mathbb{P}(I_1(S,v, u_1,\ldots, u_k)=1) \cdot \mathbb{P} (I_2 (S')=1 \mid I_1 (S,v,u_1,\ldots, u_k)=1). 
\label{eq:conditional}
\end{eqnarray}
We will provide an upper bound on $\mathbb{P} (I_2 (S')=1 \mid I_1 (S,v,u_1,\ldots, u_k)=1)$. 
If $S$ induces an $(\ell,k)$-blocking star with centre $v$ and connectors $u_1,\ldots, u_k$, 
then any spanning tree $T_{S'}$ on 
$S'$ contains the star on $S\setminus \{u_1,\ldots, u_k\}$ centred at $v$ as an induced subgraph and, moreover, one of the edges $u_i v$ is a cutting edge between $S' \setminus (S\setminus \{u_1,\ldots, u_k\})$ and $S\setminus \{u_1,\ldots, u_k\}$. Thus, $T_{S'} \setminus (S\setminus \{u_1,\ldots, u_k\})$ is a spanning tree of the subgraph induced by the 
set $S' \setminus (S\setminus \{u_1,\ldots, u_k\})$.  Furthermore, if $S'$ is a connected component in $G(n,d/n)$, there are no 
edges between $S'\setminus (S\setminus \{u_1,\ldots, u_k\})$ and $V_n \setminus S'$.  

With the above observations and $p=d/n$ we can give the following bound:
\begin{eqnarray}
\lefteqn{\mathbb{P} (I_2 (S')=1 \mid I_1 (S,v,u_1,\ldots, u_k)=1) \leq } \nonumber \\
&&\hspace{1.5cm} (r-(\ell+1))^{r-(\ell+1)-2} p^{r-(\ell+1)-1} (1-p)^{(n-r)(r-(\ell+1))}
\label{eq:cond_upper}
\end{eqnarray}
We have that 
$$(r-(\ell+1))^{r-(\ell+1)-2} p^{r-(\ell+1)-1} \leq  (rp)^{r-(\ell+1)-1} = (rp)^{r-\ell -2},$$
and for $r\leq \log n$, 
$$(1-p)^{(n-r)(r-(\ell+1))} \sim e^{-dr} \leq 2 e^{-d (r-\ell-1)} =2 e^{-d (r-\ell-1)} e^{-d}. $$
Using these in~\eqref{eq:cond_upper} we get
$$\mathbb{P} (I_2 (S')=1 \mid I_1 (S,v,u_1,\ldots, u_k)=1) \leq 2  (rp e^{-d})^{r-\ell-2} e^{-d}.$$
Hence, the left-hand side in~\eqref{eq:conditional} is bounded as 
$$ \mathbb{P} (I(S',S,v,u_1,\ldots, u_k)=1) \leq \mathbb{P}(I_1(S,v,u_1,\ldots, u_k)=1) \cdot \left(2  (rp e^{-d})^{r-\ell-2} e^{-d} \right).$$
Thus, for $n$ sufficiently large,~\eqref{eq:clk_exp} yields:
\begin{eqnarray}
\mathbb{E} (C_{\ell,k,r}) &\leq& \sum_{(S,S') \in \mathcal{S}_{\ell,k,r}} \sum_{(v,u_1,\ldots, u_k)\in S^{(k+1)}_{\not =}} 
\mathbb{P}(I_1(S,v,u_1,\ldots, u_k)=1) \cdot \left(2  (rp e^{-d})^{r-\ell-2} e^{-d} \right) \nonumber \\
&=& \sum_{S \subset V_n : |S|=\ell+k+1} \sum_{(v,u_1,\ldots, u_k)\in S^{(k+1)}_{\not =}} \mathbb{P}(I_1(S,v,u_1,\ldots, u_k)=1)  \times \nonumber \\
& &\hspace{3cm}\sum_{S' : |S'|=r, S\subset S'}2  (rp e^{-d})^{r-\ell-2} e^{-d}.  \label{eq:expckl_bound1}
\end{eqnarray} 
For a fixed choice of $S$, we will provide an upper bound on the inner sum.
This is
\begin{eqnarray*}
 \lefteqn{\sum_{S' : |S'|=r, S\subset S'}   (rp e^{-d})^{r-\ell-2} e^{-d} = 2 \sum_{r=\ell+k+1}^{\log n} 
 {n-\ell-k-1 \choose r- \ell-k-1}  (rp e^{-d})^{r-\ell-k-1} e^{-d}} \\
 &\leq& e^{-d} \sum_{r=\ell+k+2}^{\log n} \left(\frac{n e}{r-\ell -k-1} \right)^{r-\ell-k-1} \left(\frac{rde^{-d}}{n} \right)^{r-\ell-k-1} + e^{-d}((\ell+k+1) p e^{-d}) \\
 &\stackrel{d\gg 1}{=} &e^{-d} \sum_{r=\ell+k+2}^{\log n} \left(\frac{n e}{r-\ell -k-1} \right)^{r-\ell-k-1} \left(\frac{rde^{-d}}{n} \right)^{r-\ell-k-1} + o(1) \\
 &\stackrel{\frac{r}{r-\ell-k-1}\leq \ell+k+2}{\leq}& 
 e^{-d} \sum_{r=\ell+k+2}^{\log n} \left( (\ell+k+2)ed e^{-d} \right)^{r-\ell-k-1} + o(1) \\
 &=& e^{-d} \sum_{j=1}^{\log n+\ell+k+1} \left( (\ell+k+2)ed e^{-d} \right)^j + o(1) \\
 &\stackrel{(\ell+k+2)ed e^{-d} \leq e^{-d/2},  \ for \ n \ large}{\leq}& e^{-d} \sum_{j=1}^{\log n+\ell+2} \left( e^{-d/2} \right)^j + o(1) \\
 &\leq& \frac{e^{-d}}{1-e^{-d/2}} +o(1) = o(1).
\end{eqnarray*}
Note that this upper bound holds for $n$ sufficiently large uniformly over all choices of $S$. 
Thus, 
\begin{eqnarray*}
\mathbb{E} (C_{\ell,k,r}) &=& o(1) \cdot  \sum_{S \subset V_n : |S|=\ell+k+1} \sum_{(v,u_1,\ldots, u_k)\in S^{(k+1)}_{\not =}} \mathbb{P}(I_1(S,v, u_1,\ldots, u_k )=1)  = o(1) \cdot \mathbb{E} X_{\ell,k,n},
\end{eqnarray*}
which concludes the proof of the lemma. 
\end{proof}


\subsection{Small degree vertices, their structure and their role in unanimity: proof of Theorems~\ref{Spare_skewed_Stabilistion}, \ref{thm:Skewed_Majority} and~\ref{thm:Skewed_Minority}}
In this subsection, we proceed with the proof of  Theorems \ref{Spare_skewed_Stabilistion}, ~\ref{thm:Skewed_Majority} and~\ref{thm:Skewed_Minority}. Suppose that $\mathcal{I} = (G,Q,\mathcal{S})$ is an interacting node system
with $\lambda = \lambda (Q) \not = 1$ and $G=G(n,p)$.
For each $v \in V_n$ and for $\delta \in (0,1)$ we say that $v$ is $\delta$-\textit{balanced} if for all $i \in \{0 , 1\},$ we have that $\big\lvert n_{0}(v; i) - \mathbb{E}[n_{0}(v;i)] \big\rvert \leq \delta{\mathbb{E}[n_{0}(v;i)]}$. 
If $v$ is not $\delta$-balanced then we say that $v$ is $\delta$-\textit{unbalanced}, and we denote the set of $\delta$-unbalanced vertices as $\mathcal{U}_{\delta}$. We denote by $m_t = \mathrm{argmin} \{|N_t|, |P_t|\}$ 
and ${\mu}_{t} = \min\{ | N_t | ,  | P_t | \}$. The following lemma describes the first round of the evolution; notably it describes the formation of a large majority after a single round. 

Recall that $i^*\in \{0,1\}$ is the strategy which satisfies $\lambda^{1-2i^*}<1$. Note that since $\lambda \not =1$, exactly one of $0$ or $1$ satisfies this. Note further that 
$\lambda^{1-2i^*} = \min \{\lambda, \lambda^{-1}\}$. 
Hence,
\begin{equation} \label{eq:lambda_cond}
\max \{\lambda, \lambda^{-1} \} \cdot \lambda^{1-2i^*} = \max \{\lambda, \lambda^{-1} \} \cdot \min \{\lambda, \lambda^{-1} \}  = \lambda \cdot \lambda^{-1} =1.
\end{equation}
We will use this identity later on.

Now, we will show that after one round, in the majority regime strategy $i^*$ will become the dominant strategy among the vertices of $G(n,d/n)$. However, in the minority regime it will be strategy $1-i^*$ that will dominate. 
\begin{lemma}\label{Initial Sparse}
Let $p = d/n$ where $d\gg 1$.  
For any $0< \lambda \not = 1$, there exists $\gamma >0$ for which the following holds. 
A.a.s. across the product space of $G(n,p)$ and $\mathcal{S}_{1/2}$ : 
For the interacting node system $\mathcal{I} = (G(n,p), Q, \mathcal{S}_{1/2})$ with $\lambda (Q)  = \lambda$, 
we have $\mu_{1} \leq n e^{-\gamma d}$. In particular, after the first round, the majority of the vertices will be  playing either strategy $i^*$ (majority regime) or strategy $1-i^*$ (minority regime). 
Furthermore, there exists $\alpha(\lambda)>1$ such that if $d> \alpha(\lambda) \log n$, then a.a.s. $\mu_1 = 0$. 
\end{lemma}


\begin{proof}
Suppose first that our system is in the majority regime. 
By~\eqref{eq:majority}, if a vertex plays strategy $1-i^*$, then it will change strategy if 
 $n_{0}(v; 1-i^*) < \lambda^{1-2(1-i^*)} n_{0}(v ; i^*) =\lambda^{2i^* -1} n_{0}(v ; i^*) $. Also, if a vertex plays strategy $i^*$, then it will stay there, if 
 $n_0 (v;i^*) \geq \lambda^{1- 2i^*} n_0(v;1-i^*)$.  
 
Suppose now that our system is in the minority regime.  
By~\eqref{eq:minority}, if a vertex plays strategy $1-i^*$, then it will remain there if 
 $n_{0}(v; 1-i^*) \leq \lambda^{1-2(1-i^*)} n_{0}(v ; i^*) =\lambda^{2i^* -1} n_{0}(v ; i^*) $. Also, if a vertex plays strategy $i^*$, then it will switch to strategy $1-i^*$, if 
 $n_0 (v;i^*) > \lambda^{1- 2i^*} n_0(v;1-i^*)$.   
 
Now, note that 
if $v$ is $\delta$-balanced, then provided that $\delta = \delta (\lambda) \in (0,1)$ is sufficiently small
$$ \lambda^{1-2i^*} <\frac{1-\delta}{1+\delta}\leq \frac{n_{0}(v;1-i^*)}{n_{0}(v;i^*)} \leq \frac{1 + \delta}{1 -\delta} < \lambda^{2i^* -1}.$$ 
In other words, if $v$ is $\delta$-balanced, then all the above four inequalities will be satisfied. 
We thus arrive at the following conclusions: 
\begin{enumerate}
\item[1.] if $v$ is $\delta$-balanced and $Q$ is in the majority regime, then 
$S_1 (v) = i^*$;
\item[2.]  if $v$ is $\delta$-balanced but $Q$ is in the minority regime, then 
$S_1 (v) = 1-i^*$.
\end{enumerate}
Furthermore, if the majority of the vertices in $V_n$ are $\delta$-balanced, then
 $\mu_1 = |\mathcal{U}_\delta|$. 
 We will show that a.a.s. the majority of the vertices in $V_n$ are $\delta$-balanced, whereby 
 they adopt strategy $i^*$ or $1-i^*$ after one step, as described above.  

We will show that an arbitrary vertex $v\in V_n$ is $\delta$-balanced with probability $1-o(1)$. 
For any $v \in V_{n}$ the random variables $n_{0}(v;0)$ and $n_{0}(v;1)$ have identical distributions, namely the binomial distribution $\mathrm{Bin} (n-1, p/2)$. 
Therefore 
$$\mathbb{E}[n_{0}(v; 1)]=\mathbb{E}[n_{0}(v; 0)]= (1-o(1))\frac{d}{2}.$$ 
Set $\bar{\gamma} = {\delta}^{2}/7.$ 
We bound the probability that $v \in \mathcal{U}_{\delta}$.  
By Chernoff's inequality~\eqref{eq:Chernoff} we have:
$$\mathbb{P}\Big[\big\lvert n_{0}(v; 0) - \mathbb{E}[n_{0}(v;0)] \big\rvert \geq \delta{\mathbb{E}[n_{0}(v;0)]}\Big] \leq 2 e^{-\frac{\delta^2}{3} \frac{(1-o(1))d}{2} } \leq  e^{-\bar{\gamma}d}.$$
where the last inequality holds for $n$ sufficiently large. 
The same holds for $n_0 (v;1)$ as it is identically distributed to $n_0 (v;0)$.
Hence, by the union bound, for any $v \in V_n$
$$\mathbb{P} \left[v \in \mathcal{U}_\delta \right] \leq 2  e^{-\bar{\gamma}d}.$$ 
Therefore, $\mathbb{E}\left[\lvert \mathcal{U}_{\delta} \rvert \right] \leq 2n  e^{-\bar{\gamma}d}$. 

By Markov's inequality we have:
$$\mathbb{P}[\lvert \mathcal{U}_{\delta} \rvert \geq n  e^{-\bar{\gamma}d/2}] \leq \frac{2n  e^{-\bar{\gamma}d}}{n  e^{-\bar{\gamma}d/2}} =  2e^{-\bar{\gamma}d/2} = o(1).$$
Therefore, a.a.s. we have that $\lvert \mathcal{U}_{\delta} \rvert < n  e^{-\bar{\gamma}d/2}$. 
%
In turn, a.a.s. 
$\mu_1 \leq \lvert \mathcal{U}_{\delta} \rvert < n  e^{-\bar{\gamma}d/2}.$ 

Finally, note that the last inequality above implies that if $d > \alpha(\lambda) \log n$, for some $\alpha(\lambda) >1$ sufficiently large then the above is $o(1)$, which 
shows the last part of the lemma.  
\end{proof}
For the analysis of the subsequent rounds we will split the vertices of $G(n,d/n)$ into two classes and we will 
consider their evolution separately. 
More specifically, for $C \in \mathbb{N}$ we set
$$H_n(C):=H(C,G(n,d/n)) := \{v\in V_n \ : \ d (v) \geq C \}$$ 
and 
 $$ L_n(C):=L(C,G(n,d/n)) := V_n \setminus H(C,G(n,d/n)) = \{ v \in V_n \ : \ d(v) < C \}.$$

\begin{remark} \label{rem:almost_unanimity}
It is not very hard to show that in fact a.a.s. $H_n(C)$ consists of the majority of the vertices of $G(n,d/n)$ if 
$d\gg 1$. 
(For example, an application of the Chernoff bound can show that most vertices have degree around $d$.) 
More specifically, if $d \gg 1$, then by~\eqref{eq:Chernoff} for any vertex $v \in V_n$
$$ \mathbb{P} (d(v) \leq C) = o(1). $$ 
Hence, $\mathbb{E} (|L_n (C)|) = o(n)$. 
By Markov's inequality for any $\eps >0$ a.a.s.  $$|L_n (C)| < \eps n, $$
as long as $d \gg 1$. 
\end{remark}

By the previous lemma, it suffices 
to assume that $d \leq \alpha(\lambda) \log n$, as otherwise a.a.s. the process reaches unanimity after one step. 

We will now give some lemmas on the structure of the subgraph induced by the vertices in $L_n(C)$, for any 
fixed integer $C\geq 2$. Before doing this we shall give bound on the joint probability that a given collection 
of vertices $S \subset V_n$ of size $|S| = O(1)$ belong to $L_n (C)$.
\begin{claim} \label{clm:prob_small degrees}
Let $S\subset V_n$ be such that $|S| < k$, for some fixed $k\in \mathbb{N}$. Then 
 $$\mathbb{P} (\forall v \in S, d(v) \leq C) \leq  \left( 2 C \left(d e \right)^{C} \cdot e^{-d}\right)^{|S|}.$$
\end{claim} 
\begin{proof} 
If a vertex in $S$ has degree at most $C$, then it has also degree at most $C$ in $V_n \setminus S$. 
So we can write: 
$$\mathbb{P} (\forall v \in S, d(v) \leq C) \leq  \mathbb{P} (\forall v \in S, d_{V_n\setminus S}(v) \leq C).$$
Observe that these degrees form an independent family as they are determined by mutually disjoint sets of edges. 
Thereby, 
$$\mathbb{P} (\forall v \in S, d_{V_n\setminus S}(v) \leq C) = \prod_{v \in S} \mathbb{P} (d_{V_n\setminus S}(v)\leq C ). $$ 
But $d_{V_n\setminus S}(v)$ is distributed as $\mathrm{Bin} (n-|S|, d/n)$ and, therefore, its expected value 
is $d - o(1)$. Hence, $\mathbb{P} (d_{V_n\setminus S}(v) =k)$ is increasing as a function of $k$, if  $d/k \to \infty$, as $n\to\infty$.  
Using ${n \choose k} \leq \left( \frac{ne}{k}\right)^k$ we can write the following bound: 
\begin{eqnarray} \label{eq:degree_prob}
\mathbb{P} (d_{V_n\setminus S} (v) \leq C) &\leq& C \cdot {n \choose C} \left( \frac{d}{n} \right)^C 
\left(1-\frac{d}{n}\right)^{n-|S|-C} \nonumber \\
&\stackrel{|S|<k}{\leq}& C \left(\frac{n e}{C} \right)^{C} \left( \frac{d}{n} \right)^{C} e^{-d + o(1)} \nonumber \\
&\leq& 2 C \left(\frac{d e}{C} \right)^{C} \cdot e^{-d}, 
\end{eqnarray}
for $n$ sufficiently large. 
Therefore, 
$$\mathbb{P} (\forall v \in S, d_{V_n\setminus S}(v) \leq C)  \leq \left( 2 C \left(d e \right)^{C} \cdot e^{-d}\right)^{|S|}, $$
 and the claim now follows. 
\end{proof}
The following lemmas describe the structure of the subgraph induced by the vertices in $L_n(C)$.

\begin{lemma} \label{lem:distances_Ln}
Suppose that $p=d/n$ with $ c_\lambda \log n  \leq d \leq \alpha(\lambda) \log n$. 
A.a.s. there are no $\ell_\lambda+2$ vertices in $L_n(C)$ that have a
common neighbour. 
\end{lemma}
\begin{proof} 
We will use a first moment argument to bound the expected number of collections of vertices in $L_n$ of 
size $\ell_\lambda+2$ that have a common neighbour. Let $S\subset V_n$ be a subset of vertices.  
We denote the degree of $v$ outside the set $S$ by $d_{V_n\setminus S}(v)$. 
Let $S \subset V_n$ be  such that $S=\{u_1,\ldots, u_{\ell_\lambda+2}\}$ and $z \in V_n \setminus S$. 
The expected number of collections of $\ell_\lambda+2$ vertices in $L_n (C)$ which have a common neighbour 
is at most
\begin{eqnarray} 
&&{n \choose \ell_\lambda+2} \cdot (n-(\ell_\lambda+2)) \cdot \left( \prod_{i=1}^{\ell_\lambda+2} \mathbb{P} (u_i\sim z) \right)  \cdot 
\mathbb{P} \left(\forall i =1,\ldots, \ell_{\lambda+2} \ d_{V_n} (u_i) \leq C\right)
  \nonumber \\
&\stackrel{Claim~\ref{clm:prob_small degrees}}{\leq}& n^{\ell_\lambda +3} \cdot \left( \frac{d}{n}\right)^{\ell_\lambda+2} 
\left( 2 C \left(d e \right)^{C}\cdot e^{-d}\right)^{\ell_\lambda+2} \nonumber \\
& & =
n d^{\ell_\lambda+2} \cdot \left( 2 C \left(d e \right)^{C}\cdot e^{-d}\right)^{\ell_\lambda+2} 
= n \cdot e^{- d (\ell_\lambda +2) + O( \log d)}.
\label{eq:tuple_exp}
\end{eqnarray}
But $d \geq c_\lambda \log n = (\ell_\lambda+1)^{-1} \log n$. 
So $d (\ell_\lambda+2) - \log n = \Omega (\log n)$, whereby the above expected value is $o(1)$.  
\end{proof}
The above lemma implies in particular that a.a.s. at most $\ell_\lambda+1$ vertices in $L_n(C)$ are adjacent to each vertex in $H_n(C)$. 
Thus, we see that if all vertices in $H_n(C)$ play a certain strategy simultaneously, then if $C$ is 
large compared to $\ell_\lambda$, then they may stay unaffected by what the vertices in $L_n(C)$ do. 

\begin{lemma} \label{lem:light_comp}
Let $\ell \in \mathbb{N}$ and let $p =d/n$ where $\frac{1}{\ell+1} \log n \leq d=d(n) \leq \alpha(\lambda) \log n$. A.a.s. all connected sets of vertices in $L_n(C)$ 
have size at most $\ell+1$. 
\end{lemma} 
\begin{proof} 
We will show that a.a.s. there are no connected sets of vertices in $L_n(C)$ of size $\ell +2$ or more. 
If there is such a set, then in fact there must also be such a set of size exactly $\ell +2$. 
So it suffices to show that a.a.s. no such subsets exist. 

Let $S\subset V_n$ have $|S|=\ell+2$. Then
\begin{eqnarray*}
\mathbb{P} (\mbox{$S$ is connected and} \ \forall v\in S, d(v) \leq C) \leq 
\mathbb{P} (\mbox{$S$ is connected}) \cdot \mathbb{P} ( \forall v\in S, d(v) \leq C),
\end{eqnarray*}
by the FKG inequality  (Theorem~\ref{thm:FKG}), since the graph property that  $\{\mbox{$S$ is connected} \}$ is non-decreasing whereas the property that $\{\forall v\in S, d(v) \leq C\}$ is non-increasing.
Now,
$$\mathbb{P} (\mbox{$S$ is connected}) \leq |S|^{|S|-2} \cdot \left(\frac{d}{n} \right)^{|S|-1}, $$
since if $S$ induces a connected subgraph, then this has to have a spanning tree (selected in $|S|^{|S|-2}$ ways). 
By Claim~\ref{clm:prob_small degrees}, we have 
$$  \mathbb{P} ( \forall v\in S, d(v) \leq C) \leq \left( 2 C \left(d e \right)^{C} \cdot e^{-d}\right)^{|S|}.$$
Therefore, 
\begin{eqnarray*}
\mathbb{P} (\mbox{$S$ is connected and} \ \forall v\in S, d(v) \leq C) \leq
(\ell +2)^{\ell} \cdot \left(\frac{d}{n} \right)^{\ell+1} \cdot 
 \left( 2 C \left(d e \right)^{C} \cdot e^{-d}\right)^{\ell + 2}.
\end{eqnarray*}
Hence, the expected number of such subsets is at most 
\begin{eqnarray*}
O(1) \cdot {n \choose \ell +2} \cdot \left(\frac{d}{n} \right)^{\ell+1} \cdot 
 \left( d^{C} \cdot e^{-d}\right)^{\ell + 2} = n e^{-(\ell+2)d +O(\log\log n)}.
\end{eqnarray*}
But $d \geq \frac{1}{\ell+1} \log n$. Thereby, 
$(\ell+2)d - \log n = \Omega (\log n)$, and the right-hand side is $o(1)$. 
\end{proof}
\begin{lemma} \label{lem:light_trees}
Let $p =d/n$ where $c_\lambda \log n \leq d=d(n) \leq \alpha(\lambda) \log n$. A.a.s. all connected sets of vertices in $L_n(C)$ induce trees. 
\end{lemma}
\begin{proof} 
By the previous lemma it suffices only to consider sets of size at most $\ell_\lambda +1$.

Let $S\subset V_n$ with $|S|\leq \ell_\lambda+1$. Then
\begin{eqnarray*}
\mathbb{P} (\mbox{$e(S)\geq |S|$ and} \ \forall v\in S, d(v) \leq C) \leq 
\mathbb{P} (\mbox{$e(S)\geq |S|$}) \cdot \mathbb{P} ( \forall v\in S, d(v) \leq C),
\end{eqnarray*}
by the FKG inequality (Theorem~\ref{thm:FKG}), since the graph property that  $\{e(S)\geq |S| \}$ is non-decreasing and the property that $\{\forall v\in S, d(v) \leq C\}$ is non-increasing.
But
$$\mathbb{P} (e(S)\geq |S|) = O(1) \cdot \left( \frac{d}{n} \right)^{|S|}. $$
Combining this with Claim~\ref{clm:prob_small degrees} we get 
$$\mathbb{P} (\mbox{$e(S)\geq |S|$ and} \ \forall v\in S, d(v) \leq C)  = 
O(1) \cdot  \left( \frac{d}{n} \right)^{|S|} \left( d^{C} e^{-d}\right)^{|S|} = O(1) \cdot 
n^{-|S|} e^{- d |S| + O(\log \log n)}.$$ 
As there are ${n \choose |S|} =O(1) n^{|S|}$ choices for $S$, 
the expected number of such sets is 
$$ O(1) e^{- d |S| + O(\log \log n)} =o(1). $$
The lemma follows  from the union bound, taking the union over all possible values of $|S|\leq \ell_\lambda +1$. 
\end{proof}

The next lemma will help us deal with the evolution of the vertices in $H_n(C)$.  
It shows that if one of the strategies  occupies only a sublinear number of vertices in $H_n(C)$ 
(and we have almost unanimity)
then after one more round the size of the minority strategy will contain only a fraction of these vertices.
\begin{lemma}\label{Sparse_Algorithm}
Let $d=d(n)$ be such that  $1\ll d \leq  \alpha(\lambda) \log n$, and let $Q$ be a $2\times2$ non-degenerate payoff matrix. 
For any $\eps >0$ there exists $C_{\eps, \lambda} \in \mathbb{N}$ such that for any $\gamma>0$
 and for any $C \geq C_{\eps, \lambda}$ a.a.s. $G(n,d/n)$ satisfies the following: 
for any initial configuration $\mathcal{S}$ with $|m_0 \cap H_n(C)| < n e^{-\gamma d}$, 
the interacting node system 
 $\mathcal{I} = (G(n,d/n), Q, \mathcal{S})$ will have $|m_1 \cap H_n(C)| \leq  \eps |m_0 \cap H_n(C)|$.
\end{lemma}
Before we prove this lemma, let us proceed with the proofs of Theorems~\ref{Spare_skewed_Stabilistion}, 
\ref{thm:Skewed_Majority} and~\ref{thm:Skewed_Minority}. 

\begin{proof}[Proof of Theorem~\ref{Spare_skewed_Stabilistion}]
Let us first point out that by Lemma~\ref{Initial Sparse}, if $d\geq c \log n$ with
$c > \alpha(\lambda)$, where $\alpha(\lambda)$ is as in the statement of that lemma, 
then a.a.s. $\mu_1 = 0$; so the last part of Theorem~\ref{Spare_skewed_Stabilistion} follows. 
Hence, we now assume that $d \leq \alpha(\lambda) \log n$. 
We say that $G(n,d/n)$ has the \emph{minority decline property} for some $\gamma >0$, if whenever the node system 
$\mathcal{I} = (G(n,d/n), Q, \mathcal{S})$ is such that $\mu_0 \leq n e^{-\gamma d}$, then 
$|m_1 \cap H_n(C)| \leq |m_0 \cap H_n(C) |/10$. 
Observe now that if $G(n,d/n)$ has the minority decline property, 
and $\mathcal{I} = (G(n,p, Q, \mathcal{S})$ is a node system with $\mu_{0} < n e^{-\gamma d}$, for some $\gamma >0$, then the vertices in $H_n$ will reach unanimity in a finite number of rounds, by repeated applications of this definition.  

By Lemma~\ref{Initial Sparse}, a.a.s. $\mu_{1} < ne^{-\gamma d}$ for some $\gamma >0$. 
But by Lemma~\ref{Sparse_Algorithm}, if $C> C_{1/10,\lambda}$, then a.a.s. $G(n,p)$ has the minority decline property for $\gamma>0$ as above. Thus, for every 
$t\geq 1$ we have $\lvert m_t \cap H_n (C) \rvert \leq \lvert m_0 \cap H_n(C) \rvert 10^{-t}$.
So for $R = \lceil (1/\log 10) \log\left(| m_1 \cap H_n(C)|  \right)\rceil +1 = O(\log n)$ we have 
$$ \lvert m_{R} \cap H_n \rvert \leq | m_1\cap H_n(C)| 10^{-R} < 1.$$
So $|m_{R} \cap H_n (C)|=0$. 
We now show that if $C$ is sufficiently large and every vertex in $H_n(C)$ does not have too many 
neighbours inside $L_n(C)$, then once $H_n(C)$ has reached unanimity, it will stay there. 
In particular, Lemma~\ref{lem:distances_Ln} states that a.a.s. no $\ell_\lambda +2$ vertices 
in $L_n(C)$ have a common neighbour. Let us denote this event by $\mathcal{D}_n$. 
Thus, on $\mathcal{D}_n$ all vertices in $H_n (C)$ have at most $\ell_\lambda+1$ neighbours inside 
$L_n(C)$. 
\begin{claim} \label{eq:core_un} If $C\geq \max \{\ell_\lambda+4, (\ell_\lambda +1)^2\}$ 
and the vertices in $H_n(C)$ are unanimous just after step $t$, playing $i^*$ when in the majority regime, then on the event $\mathcal{D}_n$, the vertices of $H_n (C)$ will stay unanimous after step $t+1$. 
\end{claim}
\begin{proof}[Proof of Claim~\ref{eq:core_un}]
Indeed, suppose that at step $t$ all vertices of $H_n(C),$  for $C$ to be determined, are unanimous at playing strategy $i\in \{0,1\}$.
Consider a vertex $v\in H_n(C)$.
 If the event $\mathcal{D}_n$ is realised, then all but at most {$\ell_\lambda + 1 =\lceil \max \{\lambda, \lambda^{-1} \} \rceil + 1$ neighbours of  $v$ play strategy $i$. 
 So $n_t (v;i)\geq C - (\ell_{\lambda} + 1)$ and $n_t (v;1-i)\leq\lceil \max \{\lambda, \lambda^{-1}\} \rceil + 1 <  
\max  \{\lambda, \lambda^{-1}\}  + 2$.}
 
 Suppose first that we are in the majority regime. Then in this case $i=i^*$, by assumption.
Vertex $v$ will change strategy if $n_t(v;i^*) < \lambda^{1-2i^*}n_{t}(v;1-i^*)$ (cf.~\eqref{eq:majority}). 
But 
$$\lambda^{1-2i^*} n_{t}(v;1-i^{*})< \lambda^{1-2i^*}(\ell_{\lambda} + 1) \leq \lambda^{1-2i^*}( \max  \{\lambda, \lambda^{-1}\}  + 2) \stackrel{\eqref{eq:lambda_cond}}{<} 1 + 2 =3.$$ 
Therefore must have that $C - (\ell_{\lambda} + 1) < 3$. However, choosing $C \geq \ell_{\lambda} + 4$
leads to a contradiction.

If we are in the minority regime, then we want to show that $v$ will switch strategy.  
By~\eqref{eq:minority} if the entire set $H_n(C)$ plays strategy $i^*$, then $v\in H_n(C)$ will switch strategy 
if  $n_t(v;i^*) > \lambda^{1-2i^*}n_{t}(v;1-i^*)$. But $n_t (v;i^*)\geq C - (\ell_\lambda + 1)$ and $n_t(v;1-i)\leq \ell_\lambda+1$. As seen above,  if $C \geq \ell_\lambda + 4$ then  $n_{t}(v; i^*) \geq C - (\ell_\lambda + 1) \geq 3 > \lambda^{1-2i^*}(\ell_\lambda + 1) \geq n_{t}(v; 1-i^*).$
Now, if the entire set $H_n(C)$ plays strategy $1-i^*$, then $v\in H_n(C)$ will switch strategy
if  $n_t(v;1-i^*) > \lambda^{1-2(1-i^*)}n_{t}(v;i^*)$. But $n_t (v;i^*)\leq \ell_\lambda+1$ whereas 
$n_t (v;1-i^*) \geq C - (\ell_\lambda + 1)$. Thus $v$ will switch strategy if $C - (\ell_\lambda + 1) \geq \lambda^{2i^*-1}(\ell_\lambda + 1).$ Therefore, choosing $C \geq (\ell_\lambda + 1)^{2}$ yields a contradiction and completes the proof of the claim.
\end{proof}

\begin{remark}
For the minority regime, 
the above claim and Lemma~\ref{Initial Sparse} imply that, when unanimity occurs within $H_n(C)$, 
its vertices will be playing strategy $1-i^*$ at odd steps and strategy $i^*$ at even steps. 
In the majority regime, they stabilise to strategy $i^*$. 
\end{remark}

Note that by Remark~\ref{rem:almost_unanimity}, 
if $d\gg 1$, then for any fixed $C \in \mathbb{N}$ we have a.a.s. $|H_n (C)| \geq n(1-o(1))$. 
Therefore, the above analysis implies that for any $\eps>0$ there exists $\beta = \beta (\eps, \lambda)>0$ 
such that if  $d\gg 1$, then a.a.s. at least $n(1-\eps)$ vertices in $G(n,d/n)$ 
will be unanimous after at most $\beta \log n$ rounds. 
\end{proof}

\begin{proof}[Proof of Theorem~\ref{thm:Skewed_Majority}]
Suppose first that $d = c_\lambda \log n + \log \log n + \omega (n)$, where $\omega (n) \to -\infty$ 
as $n \to +\infty$. By Lemma~\ref{lem:subcritical_description} $i.$, a.a.s. there are $(\ell_\lambda,1)$-blocking 
stars in $L_1(G(n,d/n))$ that are set to the $(1-i^*,1-i^*)$-configuration. 
By Claim~\ref{clm:ell-1_maj}, those will retain this configuration forever and, therefore, they will be in disagreement with the vertices in $H_{n} (C)$.
 Hence, $u_n^{(1)}$, the probability that $L_1 (G(n,d/n))$ becomes eventually unanimous, tends to 0 as $n\to +\infty$.   

We will now consider the cases where $d = c_\lambda \log n + \log \log n + \omega (n)$, where either 
$\omega (n) \to +\infty$ or $\omega (n) \to c \in \mathbb{R}$, as $n \to +\infty$.
We will need the following claim.
\begin{claim} \label{clm:local_obstruction_maj} Let $Q$ be in the majority regime. 
If a vertex $v$ has at most $\ell_\lambda -1$ neighbours playing strategy $1-i^*$ but at least one 
playing strategy $i^*$, it will play strategy $i^*$ in the next round. 
\end{claim}
\begin{proof}[Proof of Claim~\ref{clm:local_obstruction_maj}]
If $v$ already plays strategy $i^*$, then it will change strategy if $n_t(v;i^*) < \lambda^{1-2i^*}n_{t}(v;1-i^*)$. 
But $n_t (v;i^*) \geq 1$ and $n_t(v;1-i^*) \leq \ell_\lambda -1$.
$$ \lambda^{1-2i^*}n_{t}(v;1-i^*) \leq \lambda^{1-2i^*}(\ell_\lambda-1)< \lambda^{1-2i^*} \cdot \max \{\lambda, \lambda^{-1}\} 
\stackrel{\eqref{eq:lambda_cond}}{=} 1.
$$
So $v$ will not change strategy. 
Now, if $v$ already plays strategy $1-i^*$, then it will not change its strategy if $n_t(v;1-i^*) \geq 
\lambda^{1-2(1-i^*)}n_{t}(v;i^*) = \lambda^{2i^* -1} n_t (v;i^*)$. 
But $n_t(v;1-i^*) \leq \ell_\lambda -1$ whereas $\lambda^{2i^* -1} n_t (v;i^*) \geq \lambda^{2i^*-1} > \ell_\lambda -1$. Therefore, $v$ will change its strategy into $i^*$.  
 \end{proof}
Suppose that $d= c_\lambda \log n + \log \log n + \omega (n)$ with $\omega (n) \to +\infty$ as $n\to +\infty$. 
Let $G^{(L)}(n,d/n)$ denote the subgraph of $G(n,d/n)$ induced by the vertices 
in $L_n(C)$. By Lemma~\ref{lem:light_comp} (with $\ell = \ell_\lambda$) 
and Lemma ~\ref{lem:light_trees}, every connected component of 
$G^{(L)}(n,d/n)$ is a tree of order at most $\ell_\lambda +1$. Let $T$ be one of these connected components 
that is a subgraph of $L_1 (G(n,d/n))$. 

If $|T|\leq \ell_\lambda$, then all its vertices have degree at most $\ell_\lambda -1$ in $T$. 
For $i\geq 1$, let $T^{(i)}$ denote the set of vertices in $T$
that are at distance $i$ from $H_n(C)$. Once the vertices in $H_n(C)$ have been unanimous on strategy $i^*$, they will stay there forever. 
By Claim~\ref{clm:local_obstruction_maj} the vertices in $T^{(1)}$ will adopt strategy $i^*$ and remain 
there forever. 
Assuming that the vertices $T^{(i)}$ have adopted strategy $i^*$ for ever, then the vertices of 
$T^{(i+1)}$ will adopt strategy $i^*$ too  (provided it is non-empty) by Claim~\ref{clm:local_obstruction_maj} and remain there forever. Hence, the entire vertex set of $T$ will adopt strategy $i^*$. 

Suppose now that $|T| = \ell_\lambda +1$. If all its vertices have degree at most $\ell_\lambda -1$, then eventually 
the vertices of $T$ adopt strategy $i^*$, by the above argument. 
If there is a vertex in $T$ of degree $\ell_\lambda+1$ within $T$, then $T$ must be a star. However, by 
Lemma~\ref{lem:subcritical_description} $ii.$ a.a.s. this is not an $(\ell_\lambda,1)$-blocking star. 
Thus, one of its leaves must have a neighbour in $H_n(C)$. Since it has degree 1 ($\leq \ell_\lambda -1$) inside $T$, then by 
Claim~\ref{clm:local_obstruction_maj} it adopts strategy $i^*$ after $H_n(C)$ becomes unanimous and 
stays there forever. Subsequently, the centre of $T$ will do so (it also has at most $\ell_\lambda -1$ neighbours are not playing strategy $i^*$) and finally the remaining leaves adopt it as well. 

Moreover, the above argument shows that the only connected components of $G^{(L)}(n,d/n)$ which may not adopt strategy $i^*$ are the $(\ell,1)$-blocking stars, for $\ell \leq \ell_\lambda$. 
If there are no $(\ell_\lambda,1)$-blocking stars, then $L_1 (G (n,d/n))$ will then become unanimous. 

Firstly, let us observe that if $\omega (n) \to c \in \mathbb{R}$ as $n\to \infty$, then by Lemma~\ref{lem:X_ell} $iii.$ a.a.s. there are no $(\ell_\lambda +1,1)$-blocking stars as 
$$\mathbb{P} (X_{\ell_\lambda+1, 1, n} >0)  \leq 2  e^{-(\ell_\lambda+1) \Omega (\log n)} 
= o(1).$$
Furthermore, by Lemma~\ref{lem:distances_Ln} a.a.s. there are no $\ell_\lambda+2$ vertices of degree 
1 that have a common neighbour. 
Therefore, a.a.s. there are no $(\ell,1)$-blocking stars for any $\ell \geq \ell_\lambda+1$.

Now, by Lemma~\ref{lem:subcritical_description} $ii.$, the random variable $X_{\ell_\lambda,1,n}^{(1)}$ converges in distribution as $n\to +\infty$ to a random variable distributed as $\mathrm{Po} (e^{c(\ell_\lambda +1)}/ \ell_\lambda!)$.
Thus, for any integer $k \geq 0$, we have 
$$\mathbb{P} (X_{\ell_\lambda,1,n}^{(1)} = k) \to \mathbb{P} \left(\mathrm{Po} (e^{c(\ell_\lambda +1)}/ \ell_\lambda!) = k \right) $$
as $n\to + \infty$. 

Suppose now that $X_{\ell_\lambda,1,n}^{(1)} = k$, for some $k\in \mathbb{N}_0$.  
The case $k=0$ was treated above and unanimity is attained a.a.s. (on the conditional space where 
$X_{\ell_\lambda,1,n}^{(1)} = 0$). 
Let us consider the case $k \geq 1$. 
If an $(\ell_\lambda,1)$-blocking star is initially set to $(1-i^*,1-i^*)$-configuration, then by Claim~\ref{clm:ell-1_maj} it will stay in this configuration forever. 
We thus conclude that if unanimity is achieved then no $(\ell_\lambda,1)$-blocking star is initially set to $(1-i^*,1-i^*)$-configuration. The probability of this is $(1 - 1/2^{\ell_\lambda+1})^k$. 
Also, if all $(\ell_\lambda,1)$-blocking stars attached to $L_1 (G(n,d/n))$ are initially set to $(i^*,i^*)$-configuration, they will remain so forever (cf. Claim~\ref{clm:ell-1_maj_synch}) and will be synchronised with the vertices of $H_n(C)$. Thus, $L_1 (G(n,d/n))$ will be unanimous. The probability of this 
is $1/2^{k(\ell_\lambda+1)}$. 

Consequently, 
$$\limsup_{n\to +\infty} u_n^{(1)} \leq \sum_{k=0}^{\infty} \left(1- \frac{1}{2^{\ell_\lambda+1}}\right)^k
\mathbb{P} \left(\mathrm{Po} (e^{c(\ell_\lambda +1)}/ \ell_\lambda!) = k \right).$$
and 
$$\liminf_{n\to +\infty} u_n^{(1)} \geq \sum_{k=0}^{\infty}  \left(\frac{1}{2^{\ell_\lambda+1}}\right)^k
\mathbb{P} \left(\mathrm{Po} (e^{c(\ell_\lambda +1)}/ \ell_\lambda!) = k \right). $$

Since $H_n(C)$ will reach unanimity in at most $\beta \log n$ steps, the above case analysis implies that 
$L_1(G(n,d/n))$ will reach unanimity in at most $\beta \log n + O(1)$ steps. 
\end{proof}

\begin{proof}[Proof of Theorem~\ref{thm:Skewed_Minority}] 
Let us recall that $\ell_\lambda' = \lfloor \max \{\lambda, \lambda^{-1} \}\rfloor =  \lfloor \lambda^{2i^*-1}\rfloor$.
Suppose first that $d = \frac{1}{2} \log n + \frac{1+\ell_\lambda'}{2} \log \log n + \omega (n)$, where $\omega (n) \to -\infty$ as $n \to +\infty$. By Lemma~\ref{lem:subcritical_description} $i.$ (setting $\ell=1$ and $k=\ell_\lambda'$ therein), a.a.s. there are $(1,\ell_\lambda')$-blocking stars in $L_1(G(n,d/n))$ that are initially 
set to the $(i^*,1-i^*)$-configuration. 
By Claim~\ref{clm:1-k_sub}, those will retain this configuration forever and, therefore, they will be in disagreement with the vertices in $H_{n} (C)$.
 Hence, $u_n^{(1)}$, the probability that $L_1 (G(n,d/n))$ becomes eventually unanimous, tends to 0 as $n\to +\infty$.   

Now, suppose that $d= \frac{1}{2} \log n + \frac{1+\ell_\lambda'}{2} \log \log n + \omega (n)$, 
where $\omega (n) \to +\infty$. 
As before, we let $G^{(L)}(n,d/n)$ denote the subgraph of $G(n,d/n)$ induced by the vertices 
in $L_n(C)$. By Lemma ~\ref{lem:light_comp} (with $\ell=1$) a.a.s. 
every connected component of $G^{(L)}(n,d/n)$ is of order at most 2. 
That is, a.a.s. every component of $G^{(L)} (n,d/n)$ is either a vertex or an edge.

Let $T$ be one of these connected components 
that is a subgraph of $L_1 (G(n,d/n))$. 
If $T$ is a vertex, then it will synchronise with the vertices $H_n(C)$ after the $R$\textsuperscript{th} step, where the vertices of 
$H_n(C)$ arrive at unanimity. Thus, all its neighbours (which lie in $H_n(C)$) will play the same strategy, say $i$, by~\eqref{eq:minority} this vertex will adopt strategy $1-i$ in the next round and be in agreement with the vertices of $H_n(C)$ (cf. Claim~\ref{eq:core_un}).  

Suppose now that $T$ is an edge with one of its endpoints being adjacent to vertices in $H_n(C)$. 
Hence, $T$ is a $(1,k)$-blocking star for some $k\in \mathbb{N}$. 
But in fact, $k > \ell_{\lambda}'$ as $\omega (n) \to +\infty$ and by 
Lemma~\ref{lem:subcritical_description} $iii.$
a.a.s. there are no $(1,k)$-blocking stars with $k \leq \ell_{\lambda}'$. 
Such a $(1,k)$-blocking star with $k>  \ell_{\lambda}'$,
will have its $k$ connectors inside $H_n(C)$. But recall that these will arrive at unanimity after step $R$ and will 
start alternating simultaneously between states $i^*$ and $1-i^*$.  
So by Claim~\ref{clm:1-k_sup}, the $(1,k)$-blocking star will 
synchronise with them. 

Finally, suppose that $T=v_1v_2$ is an edge where both its endpoints $v_1$ and $v_2$ 
have at least one neighbour in $H_n(C)$. 
Let $t\geq R$ be a step at which $S_t(v)= i^*$, for all $v \in H_n(C)$. Assume that $v_1$ and $v_2$ are not unanimous with $H_n(C)$. 
In particular, suppose that $S_t(v_1)=S_t(v_2)=1-i^*$. Vertex $v_1$ will not switch strategy, if $n_t(v_1; 1-i^*) \leq 
\lambda^{2i^*-1} n_t(v_1;i^*)$. But $n_t(v_1;1-i^*)=1$, $n_t (v_1;i^*) \geq 1$ and $\lambda^{2i^*-1}>1$. 
So the inequality is satisfied. The same holds for $v_2$. 
Thereby, $S_{t+1} (v_1) = S_{t+1}(v_2) = 1-i^*$ and as $S_{t+1} (v)=1-i^*$ for all $v\in H_n(C)$, thereafter $v_1,v_2$ will be synchronised with $H_n(C)$. 

Suppose now that $S_t(v_1)=1-i^*$ but $S_t (v_2) = i^*$. Then $S_{t+1} (v_1) = 1-i^*$ as $v_1$ has no neighbours who play strategy $1-i^*$. Also, $S_{t+1} (v_2) = 1-i^*$, since 
$n_{t} (v_2; i^*) > \lambda^{1-2i^*} n_t (v_2; 1-i^*)$. The latter holds since $n_{t} (v_2; i^*)  \geq 1$, $n_{t} (v_2; 1-i^*)=1$ and  $\lambda^{1-2i^*} <1$. As $S_{t+1} (v)=1-i^*$ for all $v\in H_n(C)$, thereafter $v_1,v_2$ will stay synchronised with $H_n(C)$. By symmetry, analogous argument can be used for the case $S_t(v_1)=i^*$ but $S_t (v_2) = 1-i^*$. 

\noindent 
We thus conclude that if $\omega (n) \to \infty$, then  $u_n^{(1)} \to 1$ as $n\to \infty$. 

\medskip 

Finally, suppose that $\omega (n) \to c \in \mathbb{R}$ as $n\to \infty$. 
By Lemma~\ref{lem:subcritical_description} $ii.$, the random variable $X_{1,\ell_\lambda',n}^{(1)}$ converges in distribution as $n\to +\infty$ to a random variable distributed as $\mathrm{Po} (e^{2c}/ \ell_\lambda'!)$.
Thus, for any integer $k \geq 0$, we have 
$$\mathbb{P} (X_{1,\ell_\lambda',n}^{(1)} = k) \to \mathbb{P} \left(\mathrm{Po} (e^{2c}/ \ell_\lambda'!) = k \right) $$
as $n\to + \infty$. 

Suppose now that $X_{1,\ell_\lambda', n}^{(1)} = k$, for some $k\in \mathbb{N}_0$.  
The case $k=0$ was treated above and unanimity is attained a.a.s. (on the conditional space where 
$X_{1,\ell_{\lambda}', n}^{(1)} = 0$). 
Let us consider the case $k \geq 1$. 
If an $(1,\ell_\lambda')$-blocking star is initially set to $(i^*,1-i^*)$-configuration, then by Claim~\ref{clm:1-k_sub} it will stay in this configuration forever. 
In other words, if unanimity is achieved, then no $(1,\ell_\lambda')$-blocking star is initially set to $(i^*,1-i^*)$-configuration. The probability of this is $(1 - 1/4)^k = (3/4)^k$. 
Consequently, 
$$\limsup_{n\to +\infty} u_n^{(1)} \leq \sum_{k=0}^{\infty} \left(\frac{3}{4} \right)^k
\mathbb{P} \left(\mathrm{Po} (e^{2c}/ \ell_\lambda'!) = k \right).$$

Since $H_n(C)$ will reach unanimity in at most $\beta \log n$ steps, the above argument implies that 
$L_1(G(n,d/n))$ will reach unanimity in at most $\beta \log n + O(1)$ steps. 
\end{proof}

\noindent
We conclude this section with the proof of Lemma~\ref{Sparse_Algorithm}.


\begin{proof}[Proof of Lemma \ref{Sparse_Algorithm}]

Suppose that initially the majority strategy is $i$. 
Assume first that $\mathcal{I}$ is in the majority regime. 
Then in the random graph $G(n,d/n)$ typically a vertex is expected to have many more neighbours among those playing strategy $i$ than those playing strategy $1-i$. So one would expect that most vertices will adopt strategy $i$ in the next 
round. If we revisit~\eqref{eq:majority}, we will see that if this does not happen, then $n_0(v;i) \leq \lambda^{1-2i} n_0 (v;1)$. Indeed, if $v$ initially was playing strategy $i$, it switches to $1-i$, if $n_0 (v;i) < \lambda^{1-2i} n_0 (v;1)$. If $v$ was initially playing $1-i$, then it does not switch if $n_0 (v;1-i) \geq \lambda^{1-2(1-i)} n_0 (v;i)$. 
Rearranging the latter, we also get that $n_0(v;i) \leq \lambda^{1-2i} n_0 (v;1-i)$. 

Suppose now that $\mathcal{I}$ is in the minority regime. In this case one would expect that most vertices will adopt strategy $1-i$. Suppose that a vertex $v$ initially plays $i$. By~\eqref{eq:minority}, it keeps on playing $i$ after one round, if $n_0 (v;i)\leq \lambda^{1-2i} n_0 (v;i-1)$. Similarly, if $v$ initially plays $1-i$, then it switches to 
$i$, if $n_0 (v;1-i) > \lambda^{1-2(1-i)} n_0 (v;i)$. If we rearrange the latter, we get $n_0 (v;i) < \lambda^{1-2i} n_0 (v;1-i)$. 
Furthermore, to reduce notation we set $H_n := H_n(C)$, where $C$ is to be determined later. Also, we set $L_n:=L_n(C)$. 

Assuming that initially the most popular strategy in $H_n$ is $i$, 
we say that a vertex $v \in V_n \cap H_n$ is $i$-\emph{atypical}, if $n_0 (v;i) \leq \lambda^{1-2i} n_0 (v;1-i)$. 
Let $A_n^{(i)}$ denote this set. 
We will show that a.a.s. provided that $|m_0| < n e^{-\gamma d}$ 
 $|A^{(i)}| \leq \eps |m_0 \cap H_n|$, for all partitions of $H_n$ into two parts one of which has size less than $n e^{-\gamma d}$. If this happens, then all but at most $\eps |m_0 \cap H_n|$ vertices in 
$H_n$ will behave as expected and, therefore, $|m_1 \cap H_n |\leq \eps |m_0 \cap H_n|$. 

We proceed with showing the above. 
We assume the majority strategy initially is $i$. So a vertex is $i$-atypical if $n_{0}(v; i) \leq  \lambda^{1-2i} n_{0}(v; 1-i)$. We denote the $S_{1-i} = \{v \in V_{n} : S_{0}(v) = 1-i \}.$  By assumption, we have that $\lvert S_{1-i} \rvert < n  e^{-\gamma d}$ for some $\gamma \in [0,1)$.  


We will condition on the event of Lemma~\ref{lem:distances_Ln}, which we refer to as $\mathcal{D}_n$.
That is, $\mathcal{D}_n$ denotes the event that no more than $\ell_\lambda+1$ vertices in $L_n(C)$ have a 
common neighbour.
According to Lemma~\ref{lem:distances_Ln}, we have $\mathbb{P} (\mathcal{D}_n) = 1-o(1)$.

For a partition $(U_i,U_{1-i})$ of $H_n(C)$, we assume that the vertices in $ U_j$ are assigned strategy $j$, for 
$j\in \{0,1\}$. As we pointed out previously, the event $\mathcal{D}_n$ will allow us to ignore the influence of the vertices in $L_n(C)$ on the evolution of those in $H_n(C)$, provided that $C$ is sufficiently large.
To this end, we will say that a vertex is $i$-atypical \emph{with respect to $(U_i,U_{1-i})$} 
if $n_{0}(v; i) \leq  \lambda^{1-2i} n_{0}(v; 1-i)$ for any initial assignment of strategies to the vertices of $L_n(C)$. 
We will show that a.a.s. for all configurations $(U_i,U_{1-i})$ of $H_n$ with $|U_{1-i}|< n e^{-\gamma d}$ 
there is no collection of $\eps |U_{1-i}|$ vertices in $H_n(C)$ which are $i$-atypical with respect to $(U_i,U_{1-i})$. 
This will imply that $G(n,p)$ is such that for any initial configuration $(U_i,U_{1-i})$ on $H_n(C)$, 
with $|U_{1-i}|< n e^{-\gamma d}$,  and an arbitrary configuration for the vertices in $L_n(C)$,
there will be at most $\eps |U_{1-i}|$ vertices in $H_n(C)$ that will adopt strategy $1-i$ in the subsequent round. 

We wish to bound the number of vertices  in $H_n(C)$ which are $i$-atypical with respect to a given 
configuration $(U_i,U_{1-i})$; thus we define 
$\hat{S}_i = \{ v \in U_i : n_{0}(v; 1-i) \geq \lambda^{2i-1} n_{0}(v ; i) \} $ and $\hat{S}_{1-i} = \{v \in U_{1-i} : n_0 (v;1-i) \geq \lambda^{2i-1} n_0 (v;i) \}$. 
If there are at least $\eps |S_{1-i}|$ vertices which are $i$-atypical with respect to $(U_i,U_{1-i})$, then 
either $|\hat{S}_i| \geq \eps |U_{1-i}|/2$ or $|\hat{S}_{1-i}| \geq \eps |U_{1-i}|/2$.
We will show that 
\begin{eqnarray} 
&& \mathbb{P}\left(  \bigcup_{1 \leq k < n e^{-\gamma d}}  \hspace{0.3cm}  
\vphantom{\bigcup_{1 \leq k \leq n^{\gamma}}} \bigcup_{(U_i, U_{1-i}) : \lvert U_{1-i} \rvert = k}  \left\{  |\hat{S}_i| \geq \eps |U_{1-i}|/2 \right\} \bigcap \mathcal{D}_n  \right) =o(1), \label{eq:to_prove_P_0} \\ 
&& \mathbb{P}\left(  \bigcup_{1 \leq k < n e^{-\gamma d}}  \hspace{0.3cm}  
\vphantom{\bigcup_{\frac12 \log d \leq k \leq n^{\gamma}}} \bigcup_{(U_i, U_{1-i}) : \lvert U_{1-i} \rvert = k}  
\left\{  |\hat{S}_{1-i}| \geq \eps |U_{1-i}|/2 \right\}  \bigcap \mathcal{D}_n  \right) = o(1). \label{eq:to_prove_N_0}
\end{eqnarray}
We will show that the union bound indeed suffices to show these. 
So, firstly, we will consider a fixed partition $(U_i, U_{1-i})$ as above.  
To bound $\mathbb{P} ( \lvert \hat{S}_i \rvert \geq \eps |U_i|/2 )$ and $\mathbb{P}( \lvert \hat{S}_{1-i} \rvert \geq \eps |U_i|/2 )$, 
we translate the defining conditions of $\hat{S}_i$ and $\hat{S}_{1-i}$ into a condition on the degree of these 
vertices in $U_{1-i}$. 
On the event $\mathcal{D}_n$, there are at most $\ell_\lambda+1$ neighbours of $v$ in $L_n$. 
Consider the degree of $v$ inside $U_{1-i}$, which we denote by $d_{U_{1-i}}(v)$. Similarly, 
we denote by $d_{S_{1-i}\cap L_n}(v)$ its degree inside $S_{1-i} \cap L_n$. Thus, 
$d_{U_{1-i}}(v) + d_{S_{1-i}\cap L_n}(v) = n_0 (v;0)$. But since $\mathcal{D}_n$ is realised, we have 
$d_{S_{1-i}\cap L_n}(v) \leq \ell_\lambda + 1$, whereby
\begin{equation} \label{eq:onD_n}
d_{U_{1-i}} (v) +\ell_\lambda+1 \geq  n_0 (v;0). 
\end{equation}
Furthermore, $n_0 (v;i) + n_0 (v;1-i) = d(v) \geq C$. If $\lambda^{2i-1} n_{0}(v; i) \leq n_{0}(v; 1-i)$, then 
$$ \frac{\lambda^{2i-1} +1}{\lambda^{2i-1}} n_0 (v;0) \geq C. $$
We further bound $n_0 (v;0)$ using~\eqref{eq:onD_n} and get 
$$ \frac{\lambda^{2i-1} +1}{\lambda^{2i-1}} \left(d_{U_{1-i}}(v) +\ell_\lambda+1 \right) \geq C. $$
Rearranging this we deduce that 
$$d_{U_{1-i}}(v) \geq \frac{\lambda^{2i-1}}{\lambda^{2i-1} +1}C - (\ell_\lambda+1) \geq 
\frac{\lambda^{2i-1}}{2(\lambda^{2i-1}+1)} C =:\psi_\lambda C, 
$$
provided that $C$ is large enough. 
To summarise, we have proved that if $\lambda^{2i-1} n_{0}(v; i) \leq n_{0}(v; 1-i)$ and $\mathcal{D}_n$ is realised, then 
\begin{equation} \label{eq:degree_cond} d_{U_{1-i}}(v) \geq \psi_\lambda C. 
\end{equation}

We will start with~\eqref{eq:to_prove_P_0}. 
Using~\eqref{eq:degree_cond}, we see that the event in~\eqref{eq:to_prove_P_0} is included in the following event: 
there are disjoint set sets $S, U$ with $1\leq |U| < n e^{-\gamma d}$ and $|S| = \eps |U|/2$ such that 
for any $v \in S$ we have $d_U(s) \geq \psi_\lambda C$. 

Let us consider a set $U$ with $1 \leq |U| \leq n e^{-\gamma d}$ and let $S \subset V_n \setminus U$ be such that $|S| = \eps |U|/2$. 
\begin{eqnarray*} 
\mathbb{P} (\forall v \in S,  \ d_{U}(v) \geq \psi_\lambda C) = \prod_{v \in S} \mathbb{P} (d_{U}(v) \geq \psi_\lambda C),
\end{eqnarray*} 
since these are events depending on pairwise disjoint sets of edges.

%
%
We now observe that for any $v\in S$ the random variable $d_{U}(v)$ is stochastically dominated by a random variable with distribution $\mathrm{Bin}(\lvert U \rvert, d/n)$. 
Using that $\binom{n}{k} \leq (en/k)^{k}$, 
 we can bound the above probability in the following way:
$$ \mathbb{P}(d_{U}(v)  \geq \psi_\lambda C) \leq \dbinom{\lvert U \rvert}{\psi_\lambda C} \left(\frac{d}{n}\right)^{\psi_\lambda C} \leq \left(\frac{e d  \lvert U \rvert}{\psi_\lambda  n}\right)^{\psi_\lambda C}. $$
Substituting this bound into the above inequality, we finally get
\begin{eqnarray*}
\mathbb{P} (\forall v \in S,  \ d_{U}(v) \geq \psi_\lambda C) &\stackrel{|S|=\eps |U|}{\leq}& 
\left(\frac{e d  \lvert U  \rvert}{\psi_\lambda  n}\right)^{\psi_\lambda \eps |U| C} \\
&\stackrel{d\leq \alpha(\lambda)\log n}{=}& \exp \left( -\psi_\lambda \eps |U| C \log (n/|U|) (1+o(1))\right).
\end{eqnarray*}
Now we can bound 
\begin{eqnarray*}
\lefteqn{ \mathbb{P}\left(\exists S: |S|=\eps |U|/2, \ \forall v \in S,  \ d_{U_{1-i}}(v) \geq \psi_\lambda C \right)  \leq} \\
& & {n \choose \eps |U|/2} \cdot  \exp \left( -\psi_\lambda \frac{\eps}{2} |U| C \log \left(\frac{n}{|U|} \right) (1+o(1))\right) \\
&\leq& \left( \frac{n e}{\eps |U|/2}\right)^{\eps |U|/2} \cdot  \exp \left( -\psi_\lambda \frac{\eps}{2} |U| C 
\log \left(\frac{n}{|U|} \right) (1+o(1))\right) \\
&\leq&\exp \left( -(\psi_\lambda C -1)\psi_\lambda \frac{\eps}{2} |U|  \log \left(\frac{n}{|U|} \right) (1+o(1))\right).
\end{eqnarray*}
We are now ready to show~\eqref{eq:to_prove_P_0}. 
We write
\begin{equation}\label{eq:repeated}\begin{aligned}
\lefteqn{\mathbb{P}\left(\exists U, S: S\cap U = \varnothing, \ 1\leq |U| < n e^{-\gamma d}, \ |S|=\eps |U|/2, \ \forall v \in S,  \ d_{U_{1-i}}(v) \geq \psi_\lambda C\right)\leq }  \\ 
& \sum_{1 \leq k < n e^{-\gamma d}} 
\binom{n}{k} \exp \left( -(\psi_\lambda C-1)\frac{\eps}{2} k \log \left(n/k \right) (1+o(1))\right) \\ 
&\leq \sum_{1 \leq k < n e^{-\gamma d}} 
\left( \frac{ne}{k}\right)^k \exp \left( -(\psi_\lambda C-1) \frac{\eps}{2} k \log \left(n/k \right) (1+o(1))\right) \\ 
&= \sum_{1 \leq k < n e^{-\gamma d}} 
\exp \left(k \log (n/k) (1-(\psi_\lambda C-1) \eps/2 )(1+o(1))\right) \\
&\stackrel{\log (n/k) > \gamma d}{\leq} \sum_{1 \leq k < n e^{-\gamma d}} 
\exp \left(\gamma d k  (1-(\psi_\lambda C-1) \eps/2 )(1+o(1))\right) =o(1),
\end{aligned}
\end{equation}
provided that $C$ is large enough, depending on $\eps, \lambda$.

Now, we turn to~\eqref{eq:to_prove_N_0}. Consider a partition $(U_i,U_{1-i})$ of $H_n(C)$ with 
$|U_{1-i}|$ as specified above. If $|\hat{S}_{1-i}|> \eps |U_{1-i}|/2$ and $\mathcal{D}_n$ is realised, 
then there exists a set of $\eps |U_{1-i}|/2$ vertices in $U_{1-i}$, whose degree inside $U_{1-i}$ is at least $\psi_\lambda C$. 
Hence, the total degree of this set inside $U_{1-i}$ must be at least $\frac{\eps \psi_\lambda}{2} |U_{1-i}| C$. 
In turn, the number of 
edges in $U_{1-i}$ is at least $\frac{\eps \psi_\lambda}{4}|U_{1-i}| C$. 
Thus, if $e(U_{1-i})$ denotes the number of edges inside $S_{1-i}$ we have 
$$ \mathbb{P} (|\hat{S}_{1-i} |\geq \eps |U_{1-i}|/2|, \mathcal{D}_n)\leq \mathbb{P}\left(e(U_{1-i}) \geq \frac{\eps \psi_\lambda}{4} |U_{1-i}| C\right). $$

We will show that a.a.s. any subset $U \subset V_n$ with $1\leq |U| < n e^{-\gamma d}$ has 
$e(U) \leq \frac{\eps \psi_\lambda}{4} |U| C$.
Now, $e(U)$ is stochastically dominated from above by a binomially distributed random variable 
$Y \sim \mathrm{Bin} (|U|^2, d/n)$. So 
$$  \mathbb{P} \left(e(U) \geq \frac{\eps \psi_\lambda}{4} |U| C\right) \leq \mathbb{P} \left(Y \geq  \frac{\eps \psi_\lambda}{4} |U| C\right).$$
We now bound the last probability as follows: 
\begin{eqnarray} 
\mathbb{P} \left(Y \geq  \frac{\eps \psi_\lambda}{4} |U|C \right) &\leq& {|U|^2 \choose \frac{\eps \psi_\lambda}{4} |U|C} \cdot \left( \frac{d}{n}\right)^{\frac{\eps \psi_\lambda}{4} |U|C} \nonumber \\ 
&\leq& \left( \frac{|U|^2e}{\frac{\eps \psi_\lambda}{4} |U| C} \cdot \frac{d}{n}\right)^{\frac{\eps \psi_\lambda}{4} |U|C} \nonumber \\ 
&\leq& \left( \frac{4e}{\eps \psi_\lambda} \cdot \frac{d}{C} \cdot \frac{|U|}{n}\right)^{\frac{\eps \psi_\lambda}{4} |U|C}. \nonumber 
\end{eqnarray}
Since $d = O(\log n)$, we conclude that 
$$\mathbb{P} \left(e(U) \geq \frac{\eps \psi_\lambda}{4} |U| C\right) \leq \exp \left(-(1+o(1)) \frac{\eps \psi_\lambda}{4} |U|C \log (n/|U|) \right). $$
Arguing as in the case of~\eqref{eq:to_prove_P_0}, we take the union bound over all choices 
of the subset $U$ which satisfy the assumed conditions and a similar calculation as in~\eqref{eq:repeated} 
(the only difference being that $\eps/2$ is replaced by $\eps/4$) yields:
\begin{eqnarray*}
\lefteqn{ \mathbb{P}\left(  \bigcup_{1 \leq k < n e^{-\gamma d}}  \hspace{0.3cm}  \right.  \left. \vphantom{\bigcup_{\frac{\hat{\phi}}{2}\log{n} \leq k \leq n^{\gamma}}} \bigcup_{(U_i,U_{1-i}) : \lvert U_{1-i} \rvert = k}  \left\{  |\hat{S}_{1-i}| \geq \eps |U_{1-i}|/2 \right\} \bigcap \mathcal{D}_n\right) \leq } \\
& & \mathbb{P}\left (\exists U : 1\leq |U|< n e^{-\gamma d}, \ e(U) \geq \frac{\eps \psi_\lambda}{4} |U| C \right)
 = o(1).
\end{eqnarray*}
\end{proof}


\section{Unbiased node systems in dense regimes: the $\lambda =1$ case}\label{Unskewed_Node_Systems}
\subsection{Majority and minority dynamics}


We recall that $S_{t}(v)$ is the state of a vertex $v$ at a discrete time-step $t \geq 0.$ We consider a process running on a suitably dense realisation of $G(n,p),$ and also utilise the initial configuration $\mathcal{S}_{1 / 2}$. If our interacting node system $(G(n,p), Q, \mathcal{S}_{1 / 2})$ 
is in the majority regime with $\lambda(Q) =1 $, then its evolution coincides with the \textit{majority dynamics process}. The latter is defined by the following evolution rule:
\begin{equation} \label{eq:maj_dyn}
 S_{t+1}(v) = \begin{dcases}
1 & \textrm{if} \ n_{t}(v; 1) > n_{t}(v; 0);\\
0  & \textrm{if} \ n_{t}(v; 1) < n_{t}(v; 0);\\
S_{t}(v) & \textrm{if} \ n_{t}(v; 1) = n_{t}(v; 0). \\
\end{dcases} 
\end{equation}
In other words, in the majority dynamics process, a node will always choose to adopt the state shared by the majority of its neighbours. If there is a tie, then its state will remain unchanged. 
Goles and Olivos~\cite{ar:GolOliv80} showed that majority dynamics on a finite graph becomes eventually 
periodic with period at most 2. More specifically, there is a $t_0$ 
depending on the graph such that for any $t>t_0$ and for any vertex $v$ we have 
$S_{t}(v) = S_{t+2} (v)$. 
Majority dynamics is also a special case of voting with at least two alternatives; see~\cite{ar:MosselNeemanTamuz}. 
Results on the evolution of majority dynamics on the random graph $G(n,p)$ were obtained recently by Benjamini et al.~\cite{Itai2014Unanimity}. 
The last author in collaboration with Kang, and Makai~\cite{fountoulakis2019resolution} proved the following theorem confirming the rapid stabilisation of the majority dynamics process on a suitably dense $G(n,p)$, 
confirming a conjecture stated in~\cite{Itai2014Unanimity}. Let $M_{0}$ be the most popular vertex state seen across the initial configuration.

\begin{theorem}[{\cite{fountoulakis2019resolution}}]\label{MajorityStable}
For all $\varepsilon \in [0, 1)$ there exist $\Lambda,n_{0}$ such that for all $n > n_{0}$, if $p \geq \Lambda n^{-\frac{1}{2}}$, then $G(n,p)$ is such that with probability at least $1 - \varepsilon$, across the product space of $G(n,p)$ and $\mathcal{S}_{1/2}$, the vertices in $V_n$ following the majority dynamics rule, unanimously have state $M_{0}$ after four rounds. 
\end{theorem}
We note that Theorem \ref{MajorityStable} allows us to conclude the stabilisation of the interacting node system
 with $\lambda = 1$ in the majority regime.

On the other hand, if we consider an interacting node system in the minority regime with $\lambda = 1$, then the process coincides with the \textit{minority dynamics process}, described by the following:
\begin{equation} \label{Minority_Game} 
S_{t+1}(v) = \begin{dcases}
1 & \textrm{if} \ n_{t}(v; 1) < n_{t}(v; 0);\\
0  & \textrm{if} \ n_{t}(v; 1) > n_{t}(v; 0);\\
S_{t}(v) & \textrm{if} \ n_{t}(v; 1) = n_{t}(v; 0). \\
\end{dcases} 
\end{equation}
Under these rules, nodes will update to the state shared by the minority of their neighbours. It can be readily 
checked from~\eqref{eq:majority} and~\eqref{eq:minority}, respectively, that the evolution of the above systems 
are identical to an interacting node system with $\lambda=1$. 


We show that in the minority regime unanimity is also achieved within at most four rounds too. 
However,~\eqref{Minority_Game} implies that vertex strategies will alternate synchronously with period two. 
Our theorem concerning the evolution of minority dynamics is analogous to Theorem~\ref{MajorityStable}.
\begin{theorem}\label{Minority_Unanimous}
For all $\varepsilon \in (0, 1]$ there exist $\Lambda,n_{0}$ such that for all $n > n_{0}$, if $p \geq \Lambda n^{-\frac{1}{2}}$, then $G(n,p)$ is such that with probability at least $1 - \varepsilon$, across the product space of $G(n,p)$ and $\mathcal{S}_{1/2}$, the vertices in $V_n$ following the minority dynamics rule will unanimously have the same state after four rounds. 
\end{theorem}

Throughout this section we consider $p = d/n$ where $d \geq \Lambda  \sqrt{n}$, 
for a suitably large constant $\Lambda.$ Thus, in comparison to the previous section, we will now work in a denser regime. We will comment on sparser regimes in the discussion section of our paper. 
 We consider an initial configuration of $\mathcal{S}_{1/2},$ and apply the evolution rules from~\eqref{eq:maj_dyn} in the majority regime, or the evolution rules from~\eqref{Minority_Game} in the minority regime. We refer to the node system using evolution rule~\eqref{eq:maj_dyn} as the \textit{majority game} which we denote as $\left(G(n,p), \mathcal{S}_{1/2}\right)^{>};$ while we refer to the system given by evolution rule~\eqref{Minority_Game} as the \textit{minority game}, denoted  $\left(G(n,p), \mathcal{S}_{1/2}\right)^{<}.$ We show that in the minority game, unanimity is achieved after at most four rounds with high probability.  As noted above, the majority game will give rise to stability, while the minority game produces a periodic system 
with period two. 

The quantity $\eta_{t} := \Big\lvert \lvert P_{t} \rvert - \lvert N_{t} \rvert \Big\rvert$ represents the size of the majority of the dominant strategy at time $t$. Due to the distribution of the $\mathcal{S}_{1/2}$, we have with probability $1- \varepsilon$  that the quantity  $\eta_{0} = \Big\lvert \lvert P_{0} \rvert - \lvert N_{0} \rvert \Big\rvert$ will be sufficiently bounded away from zero.

\begin{lemma}[\cite{fountoulakis2019resolution}]\label{intial_skew}

Given $\varepsilon > 0,$ set $c = c(\varepsilon) = \sqrt{2\pi} \varepsilon / 20.$ Then across the probability space $\mathcal{S}_{1/2}$,
$$\mathbb{P}\left[ \eta_{0} \geq 2c\sqrt{n}\right] \geq 1 - \varepsilon / 4, $$
for any $n$ sufficiently large.
\end{lemma}

The proof of this lemma is a direct consequence of the Local Limit Theorem (see Theorem~\ref{Local_Limit} below). We define $\mathcal{E}_{c}^+$ to be the event that $\lvert P_{0} \rvert - \lvert {N_{0}} \rvert  \geq 2c\sqrt{n},$ and $\mathcal{E}_{c}^-$ to be the event that $\lvert N_{0} \rvert - \lvert P_{0} \rvert \geq 2c\sqrt{n}.$ The events $\mathcal{E}_{c}^+$ and $\mathcal{E}_{c}^-$ occur with equal probability; therefore by symmetry 
we may condition on either $\mathcal{E}_{c}^+$ or $\mathcal{E}_c^-$, without loss of generality.

\subsection{Some results on the majority regime}
We will consider a selection of results from~\cite{fountoulakis2019resolution} (Lemmas 3.5, 3.6 therein), 
which will be useful in the minority regime analysis. The first result  concerns the expectation and variance of $n_{1}(v; 1),$ given that $\mathcal{E}_{c}^+$ has occurred.




For a vertex $v\in V_n$ let $\mathcal{N}(v) = \{ |d(v) - d | < d^{2/3}\}$. 
\begin{lemma}\label{ExpAndVar}
Consider the majority game $M = (G(n,p), \mathcal{S}_{1/2})^{>}$ where $p=d/n$. Let $c=c(\varepsilon)$ be the constant given in Lemma~\ref{intial_skew}. 
Then there exists a constant $\zeta,$ (independent of $\varepsilon)$ such that for any $v \in V_n$ and  any 
$n$ sufficiently large the following holds: for any configuration $s_0 \in \mathcal{E}_c^+$ 
and any $k\in \mathbb{N}$ such that $|k-d| < d^{2/3}$:
$$ \mathbb{E}\left[n_{1}(v; 1)  \mid \mathcal{S}_{1/2} =s_0, d(v) = k \right] \geq 
\frac{k}{2} + \frac{\zeta c}{7} {\left( \frac{d^{3}}{n} \right)}^{1/2}.$$
Moreover, there exists a positive constant $\alpha,$ such that for any $k\in \mathbb{N}$ with 
$|k-d| < d^{2/3}$ we have 
$$\mathrm{Var}\left[n_{1}(v; 1) \mid \mathcal{S}_{1/2} =s_0, d(v)=k \right] \leq \alpha d.$$
\end{lemma}

By applying Lemma \ref{ExpAndVar}, we now proceed to prove an adjustment to a result from \cite{fountoulakis2019resolution}. In the modified result, we show that with high probability a vertex $v$ will have a $n_{1}(v; 1)$ sufficiently bounded away from $d(v)/2$. In \cite{fountoulakis2019resolution}, the authors show that this quantity is at least $d (v) / 2$. However, in order for us to utilise this result for the minority game, we will instead require $n_{1}(v; 1) \geq d(v)  / 2 + 2 \gamma \sqrt{d},$ with high probability, for some positive constant $\gamma$. We elaborate on the reasoning behind this assertion in Section~\ref{Minority_Regime}.

\begin{lemma}\label{Round1Majority}
Let $\varepsilon > 0,$ and $c$ a positive constant as given in Lemma \ref{intial_skew}. 
Consider the majority game $M = (G(n,p), \mathcal{S}_{1/2})^{>}$ where $p=d/n$.
For any positive constant $\gamma$ there exists a positive constant $\Lambda = \Lambda (\gamma, \eps)$, such that for all $n$ sufficiently large if $d \geq \Lambda n^{1/2}$
and $v \in V_n,$ the followings holds: 
$$\mathbb{P}\left[n_{1}(v; 1) < \frac{d(v)}{2} + 2\gamma \sqrt{d} \ \middle\vert \ \mathcal{E}_{c}^+ \right] < \varepsilon.$$
\end{lemma}

\begin{proof}

This argument is a direct application of Chebyshev's inequality. 
Fix $s_0 \in \mathcal{E}_c^+$ and $k\in \mathbb{N}$ such that $|k-d|< d^{2/3}$.
By applying Lemma \ref{ExpAndVar} and subtracting  $2 \gamma \sqrt{d},$ from both sides we have that:
$$ \left\lvert \mathbb{E}\left[ n_{1}(v; 1) \mid \mathcal{S}_{1/2} =s_0, d(v)=k \right] - \left( \frac{d(v) }{2} + 2\gamma \sqrt{d} \right)  \right\rvert  \geq {\left( \frac{d^{3}}{n} \right)}^{1/2} \frac{\zeta c}{7} - 2\gamma \sqrt{d}.$$

By Lemma \ref{ExpAndVar} we have $\mathrm{Var}\left[n_{1}(v; 1) \mid 
 \mathcal{S}_{1/2} =s_0, d(v)=k  \right] \leq \alpha d.$ 
We now apply Chebyshev's inequality to bound the probability that  $n_{1}(v;1) < d(v)/2 + 2\gamma \sqrt{d}$. 
This gives
\begin{eqnarray*}
\lefteqn{\mathbb{P}\left[n_{1}(v; 1) < \frac{d(v)}{2} + 2\gamma \sqrt{d} \mid  \mathcal{S}_{1/2} =s_0, d(v)=k \right] \leq} \\
& & \frac{49 \alpha d }{ {\left( \zeta c \left(d^{3} / n \right)^{1/2} - 14 \gamma \sqrt{d} \right)}^2}   \leq \frac{49 \alpha}{\zeta c \left(\frac{d}{\sqrt{n}} \right) \left[\zeta c \left(\frac{d}{\sqrt{n}} \right) -28 \gamma \right]} . 
 \end{eqnarray*}

We now recall that $d > \Lambda \sqrt{n}$. If we take $\Lambda > (1 + 28 \gamma) / (\zeta c),$ then $\zeta c \left(d /\sqrt{n} \right) -28 \gamma \geq 1.$ Applying this inequality to the denominator, and choosing $\Lambda > 2 \cdot \max\{ (1 + 28 \gamma) / \zeta c \hspace{0.05cm}, 49 \alpha / \varepsilon c \zeta  \}$ we have:

\[\mathbb{P}\left[n_{1}(v; 1) < \frac{d(v)}{2} + 2\gamma \sqrt{d}  \mid   \mathcal{S}_{1/2} =s_0, d(v)=k  \right] \leq \frac{49 \alpha}{\Lambda c \zeta} \leq \varepsilon/2.\] 
Integrating over all possible choices of $s_0 \in \mathcal{E}_c^+$ and $k$ such that $\mathcal{N}(v)$ is realised, we obtain
$$\mathbb{P}\left[n_{1}(v; 1) < \frac{d(v)}{2} + 2\gamma \sqrt{d}  \mid \mathcal{E}_{c}^+, \mathcal{N}(v) \right] \leq \varepsilon/2. $$
But $\mathbb{P} \left[\mathcal{N}(v) \right] = 1-o(1)$, (which follows by a standard application of the Chernoff bound~\eqref{eq:Chernoff}), we deduce that for $n$ sufficiently large we have 
\[\mathbb{P}\left[n_{1}(v; 1) < \frac{d(v)}{2} + 2\gamma \sqrt{d}  \mid 
 \mathcal{E}_{c}^+\right] \leq \varepsilon. 
 \]
\end{proof}
\subsection{Minority Regime}\label{Minority_Regime}

We now work within the minority regime, proving Theorem \ref{Minority_Unanimous}. 
Recall that by~\eqref{Minority_Game}, a vertex will update to the state shared across the minority of its neighbours. In the event of a tie, the vertex will remain in its current state. In this section we will .
%
%
%
%

We observe similarities with Theorem \ref{MajorityStable}, namely the fact that unanimity occurs, and is achieved within at most four rounds. However the system is no longer stable, but will become periodic with period two after unanimity is reached.  

We wish to relate the proof of Theorem \ref{Minority_Unanimous} to  Theorem \ref{MajorityStable}. We first start by showing that the first round of the minority game $m = \left(G(n,p), \mathcal{S}_{1/2}\right)^{<}$ can be approximated by a specific majority game which starts on the \emph{complementary configuration}. 
For a configuration $S$ on $V_n$ we define the \emph{complementary configuration} $\bar{S}$ as follows: 
for every $v \in V_{n}$ we set $\bar{S}(v) = 1 - S(v)$. 
Suppose we have an interacting node system $\mathcal{I} = \left(G, \mathcal{S}_{1/2}\right)^{*},$ where $* \in \{ < , > \}.$ For $v \in V_{n},$ we denote $S^{\mathcal{I}}_{t}(v)$ to be the strategy of a vertex $v$ in the game $\mathcal{I}$ at time $t$; similarly we define $n^{\mathcal{I}}_{t}(v; i) = \lvert \{u : S^{\mathcal{I}}_{t}(u) = i \} \rvert$. 
While these definitions are similar to their counterparts, we would like to emphasise the role of $\mathcal{I}$, as we will generally work with two different systems: a minority game $\mathcal{I}=m$ (when $*$ is $<$) and a majority game $\mathcal{I} = M$ (when $*$ is $>$). 
We now state the following lemma concerning the approximation of a minority game $m$ to a suitably designed majority game $M$. 
\begin{lemma}\label{coupling}
Let $G=(V,E)$ be a graph and $\mathcal{S}:V\to \{0,1\}$ be a configuration. Let $m = (G, \mathcal{S})^{<}$ be a minority game, and $M = (G, \mathcal{\bar{S}})^{>}$ be a majority game with initial configuration $\bar{\mathcal{S}}.$ If for $v \in V(G)$ we have that $n_{0}^{m}(v; 0) \neq n_{0}^{m}(v; 1),$ then $S^{m}_{1}(v) = {S}^{M}_{1}(v).$ If $n_{0}^{m}(v; 0) = n_{0}^{m}(v ; 1),$ then $S^{m}_{1}(v) = 1 -{S}^{M}_{1}(v).$
\end{lemma}
 
Lemma \ref{coupling} allows us to deduce the behaviour of a significant number of vertices in the first round of the minority process. The idea is to take a minority game, complement each of the vertex strategies, and then allow one round of the evolution to occur using the majority rules. As long as a vertex satisfies the condition $n_{0}^{m}(v ; 0) \neq n_{0}^{m}(v ; 1)$, it will have the same state as if it had just evolved using the minority rules on the original configuration. We refer to the additional condition $n_{0}^{m}(v ; 0) = n_{0}^{m}(v ; 1)$, as the \textit{equal neighbourhoods condition} (\textrm{ENC}). 

\begin{proof}[Proof of Lemma \ref{coupling}]

We first assume that $n_{0}^{m}(v; 0) \neq n_{0}^{m}(v; 1).$ We split our analysis into cases which depend on  both the current state of the vertex, along with which of $n_{0}^{m}(v; 0)$ and $n_{0}^{m}(v; 1)$ is larger. In all four cases the argument is identical; we simply must show that $S^{m}_{1}(v) = S^{M}_{1}(v).$

Suppose $S_{0}^{m}(v) = 0,$ and $n_{0}^{m}(v; 0) > n_{0}^{m}(v; 1).$ By applying the minority rules, we see that $S^{m}_{1}(v) = 1$. We now consider the complementary state, $\bar{S}^m_{0}(v) = S^{M}_{0}(v) = 1.$ As we have that $n_{0}^{m}(v; 0) > n_{0}^{m}(v; 1),$ then by the definition of complementary initial configuration, it must be the case that ${n}_{0}^{M}(v; 0) < {n}_{0}^{M}(v; 1).$ By applying the majority rules to the vertex $v$ we have that ${S}^{M}_{1}(v) = 1;$ therefore, $S^{m}_{1}(v) = {S}^{M}_{1}(v)$.

For the case where $n_{0}^{m}(v; 0) = n_{0}^{m}(v; 1),$ we observe that in both games $v$ has an equal of number of the vertices playing strategy one and zero in its neighbourhood. Therefore, it follows that $S^{m}_{0}(v) = S^{m}_{1}(v)$ and ${S}^{M}_{0}(v) = {S}^{M}_{1}(v).$ However, by the definition of complementary states we have that $S^{m}_{0}(v) = 1-{S}^{M}_{0}(v),$ and thus $S^{m}_{1}(v) = 1-{S}^{M}_{1}(v).$
\end{proof}

The main application of the above lemma is to connect the games $m$ and $M$ on $G(n,d/n)$ with 
initial configuration $\mathcal{S}_{1/2}$. It is at that point where  we apply Lemma \ref{Round1Majority} to the game $m$. However, we must consider which vertices satisfy the equal neighbourhoods condition. For each vertex $v \in V_{n},$ we say that $v \in \mathrm{ENC}$ if $n_{0}^{m}(v;0) = n_{0}^{m}(v; 1)$, and we define $\mathrm{EQ}(v) = \lvert \{w : w \in \mathrm{ENC} \} \cap N_{G(n,d/n)}(v) \rvert$. Let $\gamma$ be a positive constant. We say that a vertex $v \in V_n$ has a $\gamma$-\emph{decisive neighbourhood} in $G(n,d/n)$ if $\mathrm{EQ}(v) < \gamma \sqrt{d},$ and a vertex has a $\gamma$-\textit{abundant neighbourhood} in $G(n,d/n)$ 
if $n_{1}^{M}(v;1) \geq 2\gamma \sqrt{d} + d(v) / 2.$ We say that a vertex $v\in V_n$ is $\gamma$-\textit{good} if $v$ has a $\gamma$-abundant and a $\gamma$-decisive neighnourhood in $G(n,d/n)$. 
The following corollary illustrates the role of  $\gamma$-good vertices.
\begin{corollary}\label{Round1Equivalence}
Let $\gamma > 0$. 
If a vertex $v \in V_n$ is $\gamma$-good in $G(n,d/n)$ for a given initial configuration, 
then $S^{m}_{2}(v) = 0$.
\end{corollary}
\begin{proof} Let us abbreviate $G(n,d/n)$ by $G$.
Given an initial configuration on $V_n$, 
set up the systems $m$ and $M$ as in Lemma \ref{coupling}, and assume that $v\in V_n$ is $\gamma$-good vertex. We show that $n^{m}_{1}(v ; 1) > d(v)  / 2;$ we proceed by applying Lemma~\ref{coupling} to bound $n_{1}^{m}(v; 1)$ from below. This lemma implies that if $u \notin \mathrm{ENC}$, then $S_{1}^{m}(u) = S_{1}^{M}(u)$. Hence, for all $u \in N_{G}(v) \setminus \mathrm{ENC},$ if $S^{M}_{1}(u) = 1$, then $S^{m}_{1}(u) = 1.$ However for $w \in \mathrm{ENC} \cap N_{G}(v),$ we have that $S_{1}^{m}(w) = 1 - S_{1}^{M}(w).$ If we assume that for all $w \in \mathrm{ENC} \cap N_{G}(v)$ we have that $S_{1}^{M}(w) = 1,$ then we would minimise the size of $n^{m}_{1}(v;1)$ as in that case $S_{1}^{m}(w) = 0.$ As a direct consequence, we have that $n^{m}_{1}(v; 1) \geq n^{M}_{1}(v ; 1) - \mathrm{EQ}(v).$ Therefore, by combining this bound with the definitions of $\gamma$-abundance and $\gamma$-decisiveness, we have: 

\begin{eqnarray*} n^{m}_{1}(v; 1) &\geq& n^{M}_{1}(v; 1) - \mathrm{EQ}(v)   > n^{M}_{1}(v; 1) -  \gamma \sqrt{d}
\geq \frac{d(v)}{2} + 2\gamma \sqrt{d} - \gamma\sqrt{d} \\
&=& \frac{d(v)}{2} + \gamma \sqrt{d}. \qedhere 
\end{eqnarray*}
 
\end{proof}

\subsection{Bounding the size of $\mathrm{EQ}(v)$} In light of Corollary \ref{Round1Equivalence}, our aim is to show that there are a significant number of good vertices. We first show that with high probability, there are a sufficient number of vertices with a $\gamma$-decisive neighbourhood. The proof of this bound will require us to invoke the Local Limit Theorem for sums of Bernoulli-distributed random variables, which follows from Theorem VII.6 p.197 in~\cite{petrov2012sums}. We state this result as follows.
\begin{theorem}[\cite{petrov2012sums}]\label{Local_Limit}
Let $X_1, \ldots X_n$ be independent identically distributed Bernoulli-distributed random variables with $\mathbb{E}(X_1)=\mu$ and $\mathrm{Var} (X_1) =\mu -\mu^2=: \sigma^2 >0$. 
Let also $X = \sum_{i=1}^{n} X_{i}$. There exists $\rho$ depending on $\mu$ for which:
$$ \sup_{i \in \mathbb{N}_{0} } \left\lvert \sqrt{\mathrm{Var}[X]} \cdot \mathbb{P}[X = i] - \frac{1}{2\pi} \exp \left( - \frac{ {\left( i - \mathbb{E}[X] \right)}^2}{2 \mathrm{Var}[X]} \right) \right\rvert < \frac{\rho}{\sqrt{\mathrm{Var}[X]}}.$$
\end{theorem}

The following lemma provides an upper bound  on the number of vertices in the neighbourhood of $v\in V_n$ in 
$G(n,p)$ which satisfy the equal neighbourhoods condition. 
\begin{lemma}\label{equal_neighbourhoods}
Consider $ G(n, p)$ with $np=:d = d(n) \to \infty$ as $n\to \infty$ and let $v \in V_n$.  
For every $\varepsilon > 0,$ there exist positive constants $\gamma$ and $n_0$ such that for all $n>n_0$
we have
$$\mathbb{P}[ \mathrm{EQ}(v) \geq \gamma \sqrt{d} ] < \varepsilon.$$  
\end{lemma}
\begin{proof}
For a vertex $v \in V_n$ we bound the value of $\mathbb{E}[\mathrm{EQ}(v)]$. Without loss of generality we may also assume that $S_{0}(v) = 1$. For a vertex $w\in V_n$ to belong to $\mathrm{EQ}(v),$ we must have that $w \in N_{G(n,p)}(v)$ and $w \in \mathrm{ENC}.$ Therefore, we write 
\begin{eqnarray*}
\mathbb{E}[\mathrm{EQ}(v) \mid S_{0}(v) = 1] &=& \mathbb{E}\left[\sum_{w :  w \neq v} \mathbbm{1}_{ \{w \sim v\} } \mathbbm{1}_{\{w \in \mathrm{ENC} \}} \mid S_{0}(v) = 1\right] \\
&=& \sum_{w :  w \neq v} \mathbb{P} \left[ \{ w \sim v \} \cap \{w \in \mathrm{ENC}\} \mid S_{0}(v) = 1\right].
\end{eqnarray*}
By conditioning on the event $\{ w \sim v \},$ we have the following:
\begin{eqnarray*}
\mathbb{E}[\mathrm{EQ}(v)\mid S_{0}(v) = 1] &=& \sum_{w : w \neq v  } \mathbb{P}[w \sim v ]\mathbb{P}[w \in \mathrm{ENC} \mid 
 w \sim v, S_{0}(v) = 1] \\
  &=& \frac{d}{n}\sum_{w :  w \neq v}\mathbb{P}[w \in \mathrm{ENC} \mid w \sim v, S_{0}(v) = 1 ].
 \end{eqnarray*}
We now condition on the size of $N_{G(n,p)}(w) \backslash \{ v \}.$ The size of  $N_{G(n,p)}(w) \backslash \{ v \}$ can vary from zero to $n-2.$ We set $\mu' = \mathbb{E}\left[\mathrm{EQ}(v)\right]$ and let 
$\mathcal{N}_{k}(w)$ be the event that $| N_{G(n,p)}(w) \setminus \{v \} | = k$. 
Note that for $w \in \mathrm{ENC}$, we necessarily have $k$ is odd.  
Applying the law of total probability and bounding the upper and lower extremes of $k,$ we have that:
\begin{eqnarray}\lefteqn{\mathbb{P}[w \in \mathrm{ENC} \mid w \sim v, S_{0}(v) = 1] \leq  \mathbb{P}\bigg[ \Bigl\lvert \left\lvert N_{G(n,p)}(w) \backslash \{v \} \right\rvert - \mu' \Bigr\rvert  \geq d^{3/4} \bigg] +} \nonumber \\ 
&&\sum_{k \ :\  |k - \mu'| \leq d^{3/4}, k \ odd}\mathbb{P}\big[ \lvert N_{G(n,p)}(w) \setminus \{v\} \rvert = k\big] \mathbb{P}\left[ w \in \mathrm{ENC} \ \Big\vert \{ w \sim v \} \cap \mathcal{N}_{k}(w) \cap \{S_{0}(v) = 1 \} 
 \right]. \nonumber \\
 & & \label{eq:enc}
\end{eqnarray}


 As we have conditioned on $\{ v \sim w \},$ we have that $\lvert N_{G(n,p)}(w) \backslash \{v \} \rvert \sim \mathrm{Bin}(n-2, d /n).$ So, $\mathrm{Var}[\lvert N_{G(n,p)}(w) \backslash \{v \} \rvert] = (n-2)(d/n)(1-d/n)$ and therefore  $\mathrm{Var}[\lvert N_{G(n,p)}(w) \backslash \{v \} \rvert] <  d$. By Chebyshev's inequality,
 \begin{equation}\label{eq:cheb}
 \mathbb{P}\bigg[ \Big\lvert \left\lvert N_{G(n,p)}(w) \backslash \{v \} \right\rvert  - \mu' \Big\rvert  \geq d^{3/4} \bigg] \leq \frac{d}{d^{6/4}} = \frac{1}{\sqrt{d}}. 
 \end{equation}
Recall that $\mathcal{N}_{k}(w)$ is the event that $\lvert N_{G(n,p)}(w) \backslash \{v \} \rvert = k$.
It follows from~\eqref{eq:cheb} that:
\begin{eqnarray*}
\lefteqn{\mathbb{E}[\mathrm{EQ}(v) \mid S_{0}(v) = 1 ] \leq} \\
&& \frac{d}{n} \sum_{w :  w \neq v} \left( \frac{1}{\sqrt{d}} +
\sum_{k:  |k -\mu' | \leq d^{3/4}, k \ odd} \mathbb{P}\big[ \mathcal{N}_{k}(w) \big] \mathbb{P}\left[ w \in \mathrm{ENC} \Big\vert  \{ w \sim v \} \cap
\mathcal{N}_{k}(w) \cap \{S_{0}(v) = 1\}
 \right]  \right). 
 \end{eqnarray*}

We now turn our attention to the range where $\mu - d^{3/4} \leq k \leq \mu + d^{3/4}.$ We wish to apply Theorem \ref{Local_Limit} to bound $\mathbb{P}\left[w \in \mathrm{ENC} \mid \{ w \sim v \} \cap \mathcal{N}_{k}(w)\cap \{S_{0}(v) = 1\}\right].$ For a vertex $w \in N_{G(n,p)}(v),$ we define $X_{w}^{+} = \Big\lvert \{ u \in N_{G(n,p)}(w) \backslash \{v \} : S_{0}(u) = 1  \} \Big\rvert.$  As we have conditioned on $v \sim w$ and $S_{0}(v) = 1$, for odd $k$ the following holds:
\begin{eqnarray}  
\lefteqn{\mathbb{P}\left[ w \in \mathrm{ENC} \ \middle\vert \ \{ w \sim v \} \cap \mathcal{N}_{k}(w)
\cap \{ S_{0}(v) = 1\}\right]=} \nonumber \\ 
& &  \mathbb{P} \left[ X_w^{+} = (k-1)/2  \ \Big\vert \ \{ w \sim v \} \cap \mathcal{N}_{k}(w) \cap \{S_{0}(v) = 1 \}\right]. \label{eq:enc_prob}
\end{eqnarray}

Now for a given $k,$ suppose we condition on the event $\mathcal{N}_{k}(w).$ By conditioning on $\mathcal{N}_{k}(w)$, we observe that $X_{w}^{+} \sim \mathrm{Bin}(k , 1/2)$; so $\mathrm{Var}[{X_{w}^{+}}] = k/4.$ We now apply Theorem \ref{Local_Limit}, to bound the probability of the event given by~\eqref{eq:enc_prob}: there exits a constant $\rho$ such that for all $k$ we have:
\begin{eqnarray}
\lefteqn{\mathbb{P} \left[ X_{w}^{+} = \frac{k-1}{2} \middle\vert \ \{ w \sim v \} \cap \mathcal{N}_{k}(w) \cap \{ S_0(v)=1\} \right] \leq } \nonumber \\
& & \frac{\rho}{\mathrm{Var}\left[ X_{k}^{+}\right]} + \frac{1}{\sqrt{2 \pi \mathrm{Var}\left[ X_{k}^{+}\right]}}e^{-1 / \left( 8 \mathrm{Var}\left[ X_{k}^{+}\right] \right)}
 \leq   \frac{2 \rho}{ \sqrt{\mathrm{Var}\left[ X_{k}^{+}\right]}} = \frac{4 \rho}{\sqrt{k}} . \label{eq:LLT_result}
\end{eqnarray}
Using~\eqref{eq:enc_prob} and~\eqref{eq:LLT_result} in~\eqref{eq:enc}, we finally deduce that
\begin{eqnarray*} 
\lefteqn{\mathbb{E}[\mathrm{EQ}(v)\mid S_0(v)=1]  \leq}\\
&&  \frac{d(n-1)}{n}  \left(\frac{1}{\sqrt{d}} + \max_{k  \ : \  \lvert k - \mu' \rvert \leq  d^{3/4}, k \ odd} \left\{ \mathbb{P} \left[ X_{k}^{+} = (k-1)/2  \ \Big\vert \ \{ w \sim v \} \cap \mathcal{N}_{k}(w)  \cap \{S_0(v)=1\}\right]\right\} \right). 
\end{eqnarray*}
Using the bound~\eqref{eq:LLT_result}, the above expression becomes
$$\mathbb{E}[\mathrm{EQ}(v)\mid S_0 (v) =1] \leq \frac{d(n-1)}{n}\left(\frac{1}{\sqrt{d}} + \frac{4 \rho}{ {\sqrt{ \mu - d^{3/4}}}} \right).$$
A similar argument gives the same upper bound on $\mathbb{E}[\mathrm{EQ}(v)\mid S_0 (v) =0]$.

Since $\mu' = \Theta(d),$ therefore there exists a constant $\kappa,$ such that for sufficiently large $d$ we have that $\mu' - d^{3/4} \geq \kappa d.$ Letting $\Gamma = 2( \sqrt{\kappa} + 4 \rho) / \sqrt{\kappa}$, 
for sufficiently large $d$ we have that:
$$\mathbb{E}[\mathrm{EQ}(v)] \leq \frac{\Gamma d }{\sqrt{d}} = \Gamma \sqrt{d}.$$
Fix $\varepsilon > 0$ and define $\gamma = \Gamma / \varepsilon.$ Then by Markov's inequality 
\[ \mathbb{P} \left[ \mathrm{EQ}(v) \geq \gamma \sqrt{d}\right] <  \frac{\mathbb{E}[\mathrm{EQ}(v)]}{\gamma} \leq \varepsilon. \qedhere \]
\end{proof}

\subsection{A bound on the number of good vertices} We now have a suitable bound on $\mathrm{EQ}(v),$ which holds with high probability. Furthermore, we would like to show that there are a significant number of good vertices. We denote the set of $\gamma$-good vertices in $G(n,p)$ as $\mathrm{GD}_\gamma$. We recall that a vertex $v\in V_n$ of $G(n,p)$ is $\gamma$-good if  $v$ has a $\gamma$-abundant and a $\gamma$-decisive neighbourhood. The following corollary asserts that our node system will have a significant number of good vertices. We also recall the definition of the definition of the event $\mathcal{E}_{c}^-$: we say that the event $\mathcal{E}_{c}^-$ occurs if $N_{0} - P_{0} \geq 2c \sqrt{n}$, where $P_{0} = \sum_{v \in V_{n}} S_{0}(v)$ and $N_0 = n - P_0$. 
\begin{corollary}\label{TwoRoundZero}
Let $\varepsilon > 0$ and $m = (G(n,d/n), \mathcal{S}_{1/2})^{<}$, where 
$d > \Lambda n^{1/2}$. Suppose that $\mathcal{E}_{c}^-$ has occurred for some constant $c>0$. Then there exist positive constants $\Lambda$ and $ \gamma$, such that for all $n$ large enough we have with probability at least $1 - \varepsilon$ that $ \lvert \mathrm{GD}_\gamma \rvert \geq n(1 - \varepsilon).$
\end{corollary}
\begin{proof}
We fix $\varepsilon > 0$ and again define the  majority game $M = (G(n,d/n), \bar{\mathcal{S}}_{1/2})^{>}$ as above. As we have conditioned on $\mathcal{E}_c^-$ in $m,$ then $M$ starts with an initial configuration which satisfies $\mathcal{E}_c^+$. By Lemma~\ref{equal_neighbourhoods}, there exists a constant $\gamma$ such that for all $n$ sufficiently large and all $v \in V_{n},$ we have that $\mathbb{P}\left[\mathrm{EQ}_\gamma (v)  \geq \gamma \sqrt{d} \right] < \varepsilon^{2} / 2$. Also, by Lemma~\ref{Round1Majority}, we may select $\Lambda$ large enough, such that for all $n$ sufficiently large and $v \in V_{n}$, we have that $\mathbb{P}\left[ n^{M}_{1}(v; 1) < d(v) / 2  + 2 \gamma \sqrt{d} \mid \mathcal{E}_c^+ \right] < \varepsilon^{2} / 2$.  

Now, we denote the events that $v$ has a $\gamma$-decisive neighbourhood by $D_{\gamma}(v)$ and that 
$v$ has an $\gamma$-abundant neighbourhood by $A_{\gamma}(v).$  Note that  $D_\gamma^{c} (v)$ and $A_\gamma^{c} (v)$ are the events of the previous paragraph for this particular $\gamma$. 
By the union bound, we have that:
$$\mathbb{P}[v \notin \mathrm{GD}_\gamma] = \mathbb{P}[A_\gamma^{c}(v) \cup D_\gamma^{c} (v)] \leq \mathbb{P}[A_\gamma^{c}(v)] + \mathbb{P}[D_\gamma^{c} (v)] < \varepsilon^{2}.$$

By taking a sum across all vertices, and applying the above inequality we have that $\mathbb{E}\left[ \lvert V_{n} \setminus \mathrm{GD}_\gamma \rvert \right] \leq \varepsilon^{2}n,$ where $\lvert V_{n} \setminus \mathrm{GD}_\gamma \rvert$ is the number of vertices which are not good. By Markov's inequality, we have that $\mathbb{P}[ \lvert V_{n} \setminus \mathrm{GD}_\gamma \rvert > \varepsilon n] \leq \varepsilon.$ Therefore, with probability at least $1 - \varepsilon,$ we have that $\lvert \mathrm{GD}_\gamma \rvert \geq n ( 1 - \varepsilon).$  
\end{proof}
Now, if there are at least $(1 - \varepsilon)n$ good vertices (which occurs with high probability), then by Corollary~\ref{Round1Equivalence} we have that $\lvert N_{2} \rvert \geq (1- \varepsilon)n$.
 
\subsection{The final two rounds}

 In the remaining two rounds we claim that unanimity will occur.  We note that as a consequence of Corollary \ref{TwoRoundZero}, at time $t = 2$ our node system satisfies the hypothesis  of Lemma \ref{Sparse_Algorithm}. However, as we are working within a denser regime, it turns out that 
 unanimity can be reached in the entire $G(n,p)$ much faster than $\beta \log{n}$ rounds. 
 We state a result from \cite{fountoulakis2019resolution}, which concerns the rapid formation of unanimity for majority dynamics, wherein one of the initial strategies has a linear majority.  We recall the definitions $P_{t} = \{ v \in V_{n} : S_{t}(v) = 1 \},$ and $N_{t} = \{ v \in V_{n} : S_{t}(v) = 0 \}.$ Suppose that $\mathcal{S}$ is an initial configuration of vertex states on $V_{n},$ and $\delta\in (0,1)$ is fixed. We say that $\mathcal{S} \in \hat{\mathcal{S}_{\delta}},$ if we have that $\lvert N_{0} \rvert \geq (1 - \delta)n$ in $\mathcal{S}.$ Thus $\hat{\mathcal{S}_{\delta}}$ is the collection of all initial vertex configurations where there are at least $( 1- \delta)n$ vertices in the zero state. We now consider the following result considering the two round evolution of games utilising initial states belonging to $\hat{\mathcal{S}_{\delta}}$ for sufficiently small $\delta.$

\begin{lemma}\label{Final2RoundEquiv}
Let $p \geq \Lambda n^{-\frac{1}{2}}$ and $\delta < 1 /10.$ On the probability space of $G(n,p)$ we have that a.a.s. for any $\mathcal{S} \in \mathcal{S}_{\delta}$ the majority game $(G(n,p) , \mathcal{S})^{>}$ will reach  unanimity after at most two rounds.
\end{lemma}
\begin{proof}[Proof of Theorem~\ref{Minority_Unanimous}]
While the above lemma directly concerns the majority dynamics game, with only some minimal adjustments to the argument, we may apply it to the minority game as well. We fix a minority game $(G, \mathcal{S})^{<},$ where $\mathcal{S} \in \hat{\mathcal{S}_{\delta}}$. By following the proof given in \cite{fountoulakis2019resolution}, after one round of applying the minority game rules we have with probability $1 - o(1)$ that $\lvert N_{1} \rvert \leq d / 10$. However a standard application of the Chernoff bound can show that a.a.s. the minimum degree of $G(n,p)$ is at least $d/2$. Therefore for all $v \in V_{n}$ we have that $n_{1}(v; 1) > n_{1}(v; 0)$. In turn, we deduce that $S_{2}(v) = 0$ for all $v \in V_{n}$. As unanimity is achieved, it is clear that the system will also display periodic behaviour of period two from the evolution rules of the minority dynamics~\eqref{Minority_Game}. 
\end{proof}


\section{Discussion}

In this paper we study the evolution of games on $G(n,p)$ under the best response rule. Our first result concerns the rapid formation of unanimous strategies on the node system $(G(n,p), Q, \mathcal{S}_{1/2})$ for $p \geq \Lambda n^{-1/2}.$ As a byproduct of this analysis, we also prove an analogous result of Fountoulakis et al. \cite{fountoulakis2019resolution} regarding the rapid formation of unanimity in the random graph minority game. Our second main result concerns the formation of unanimous strategies in sparser regimes, namely beyond the connectivity threshold, and on the presence of a strategic bias given by $\lambda \neq 1.$ A natural question of the sparse regime is to ask whether we can remove the condition that $\lambda \neq 1.$ As previously seen, the case $\lambda = 1$ reduces to the problem of majority and minority dynamics. However the study of majority dynamics for $p = o\left(n^{-1/2}\right)$ imposes immense complications. 
Little is currently known about how these systems evolve apart from becoming eventually periodic. 
However the following was conjectured by Benjamini et al.~\cite{Itai2014Unanimity}.
\begin{conjecture}[\cite{Itai2014Unanimity}]
With high probability over choice of random graph and choice of initial state, if $p \geq d/n$ then the following holds:
\begin{enumerate}
    \item If $d \gg 1$ then, for any $\varepsilon > 0$ and $n$ sufficiently large we have,
    
    $$\lim_{t \rightarrow \infty} \hspace{0.05cm} \Big\lvert \lvert P_{2t} \rvert - \lvert N_{2t} \rvert \Big\rvert \in \left[(1-\varepsilon)n, n\right].$$ 
    
    \item If $d$ is bounded then for any $\varepsilon > 0$ and for $n$ sufficiently large,
    
    $$ \lim_{t \rightarrow \infty} \hspace{0.05cm} \Big\lvert \lvert P_{2t} \rvert - \lvert N_{2t} \rvert \Big\rvert  \in [(1- \varepsilon/2)n, (1+\varepsilon/2)n].$$
\end{enumerate}
\end{conjecture}

These results suggest the idea of long term almost unanimity, when $d\gg 1$. 
Theorem \ref{Spare_skewed_Stabilistion} can be thought as an approximate approach to studying these kinds of systems. By introducing the strategic bias $\lambda \not =1$, we are able to show a stronger form of (1) in the above conjecture. In fact, we have been able to identify precisely the critical density of the random graph 
around which its largest connected component (which contains the overwhelming majority of its vertices) 
achieves unanimity. Furthermore, non-unanimity occurs, we identified those substructures which play different 
strategies to the majority of the vertices in this component. 
However, it is still unclear how to approach these sparser interacting node systems when $\lambda = 1.$ 

Another direction for further study is the setting where the game played by the players has more than 2 strategies. 
In this case, it is not clear whether one of the strategies prevails eventually or we have the co-existence of two or more strategies. 

\appendix

\section{Derivation of evolution rules}

In this section we briefly detail the derivation of equations $(1)$ and $(2).$ We recall that if an agent chooses strategy $i$ in round $t,$ then the incentive for them to switch to strategy $1-i$ in round $t+1,$ is given by the following condition:

$$
  n_{t}(v; 0) \left(q_{i , 0} - q_{ 1 - i, 0} \right) <  n_{t}(v; 1)\left(q_{ 1 - i , 1} - q_{i , 1}\right).$$
  
As we have assumed that $Q$ is non-degenerate, we have two cases: Either $q_{0,0} > q_{1,0}$ and $q_{1,1} > q_{0,1},$ or $q_{0,0} < q_{1,0}$ and $q_{1,1} < q_{0,1}.$
Suppose that the former case occurs, then we must check the incentive for $i=0$ an $i=1.$ Suppose $S_{t}(v) = 0,$ then the incentive for changing strategy is given as follows:

$$n_{t}(v; 0) \left(q_{0 , 0} - q_{ 1, 0} \right) <  n_{t}(v; 1)\left(q_{ 1  , 1} - q_{0 , 1}\right).$$

We now re-arrange the above, and recall that $\lambda = (q_{1,1} - q_{0,1})/(q_{0,0}-q_{1,0}),$ to form the first evolution rule: 

\begin{equation}n_{t}(v;0) < n_{t}(v;1) \frac{q_{1,1} - q_{0,1}}{q_{0,0}-q_{1,0}} = n_{t}(v;1) \lambda. \tag{$\mathrm{A}1$} \end{equation}

Similarly if we instead have that $S_{t}(v) = 1,$ then the incentive is expressed as:

$$ n_{t}(v; 0) \left(q_{1 , 0} - q_{ 0, 0} \right) <  n_{t}(v; 1)\left(q_{ 0  , 1} - q_{1 , 1}\right). $$

We recall that in this case we have $q_{0,1} < q_{1,1},$ and therefore we re-arrange as follows:

\begin{equation} n_{t}(v;1) < n_{t}(v;0) \frac{q_{1,0} - q_{0,0}}{q_{0,1} - q_{1,1}} = n_{t}(v;0)\frac{q_{0,0} - q_{1,0}}{ q_{1,1} - q_{0,1}} = n_{t}(v;0) \frac{1}{\lambda}. \tag{$\mathrm{A}2$} \end{equation}

Combing equations $(\mathrm{A}1)$ and $(\mathrm{A}2)$ provides the required evolution rules as described by $(1)$. For the second case, an identical argument will produce the rules given by $(2)$.

\bibliography{references}{}

\bibliographystyle{plain}

\end{document}